\documentclass[review,hidelinks,onefignum,onetabnum]{siamart220329}

\usepackage{amsmath,amssymb,subfigure}      
\usepackage{booktabs,multirow,threeparttable}  
\usepackage{lipsum}
\usepackage{amsfonts}
\usepackage{graphicx}
\usepackage{epstopdf}
\usepackage{algorithmic}
\usepackage{mathrsfs} 
\ifpdf
  \DeclareGraphicsExtensions{.eps,.pdf,.png,.jpg}
\else
  \DeclareGraphicsExtensions{.eps}
\fi

\newsiamremark{remark}{Remark}
\newsiamremark{hypothesis}{Hypothesis}
\crefname{hypothesis}{Hypothesis}{Hypotheses}
\newsiamthm{claim}{Claim}

\headers{A Primal SDG Method on Polytopal Meshes}{L. Chen, X. Huang, E. Park and R. Wang}

\title{A Primal Staggered Discontinuous Galerkin Method on Polytopal Meshes\thanks{Submitted to the editors DATE.
\funding{The first author was supported by DMS-2309777 and DMS-2309785. The second author was supported by the National Natural Science Foundation of China Project 
12171300. The third author was supported by the National Research Foundation of Korea (NRF) grant funded by the Ministry of Science and ICT (NRF-2022R1A2B5B02002481). The fourth author was supported by the National Key Research and Development Program of China (grant No. 2020YFA0713602, 2023YFA1008803), and the Key Laboratory of Symbolic Computation and Knowledge Engineering of Ministry of Education of China housed at Jilin University.
}}}

\author{
Long Chen
\thanks{Department of Mathematics,
University of California at Irvine, 
Irvine, CA 92697 USA
(\email{chenlong@math.uci.edu})}
\and
Xuehai Huang
\thanks{School of Mathematics, 
Shanghai University of Finance and Economics, 
Shanghai 200433, China 
(\email{huang.xuehai@sufe.edu.cn})}
\and
Eun-Jae Park
\thanks{Department of Computational Science and Engineering, 
Yonsei University, Seoul 03722, Korea
(\email{ejpark@yonsei.ac.kr})}
\and
Ruishu Wang
\thanks{School of Mathematics,
Jilin University, Changchun,
Changchun, Jilin 130012, China
(\email{wangrs\_math@jlu.edu.cn})}}

\usepackage{amsopn}

\renewcommand{\div}{\operatorname{div}}

\newcommand{\dx}{\,{\rm d}x}

\newcommand{\dd}{\,{\rm d}}
\newcommand{\Oplus}{\ensuremath{\vcenter{\hbox{\scalebox{1.5}{$\oplus$}}}}}

\graphicspath{{./figures/}{../figures/}}

\ifpdf
\hypersetup{
  pdftitle={A Primal Staggered Discontinuous Galerkin Method on Polytopal Meshes},
  pdfauthor={L. Chen, X. Huang, E. Park and R. Wang}
}
\fi

\begin{document}

\maketitle

\begin{abstract}
This paper introduces a novel staggered discontinuous Galerkin (SDG) method tailored for solving elliptic equations on polytopal meshes. Our approach utilizes a primal-dual grid framework to ensure local conservation of fluxes, significantly improving stability and accuracy. The method is hybridizable and reduces the degrees of freedom compared to existing approaches. It also bridges connections to other numerical methods on polytopal meshes. Numerical experiments validate the method's optimal convergence rates and computational efficiency.
\end{abstract}

\begin{keywords}
staggered discontinuous Galerkin method, primal-dual grid, polytopal meshes, inf-sup condition, error estimate
\end{keywords}

\begin{MSCcodes}
65N30, 65N15, 65N12
\end{MSCcodes}

\section{Introduction}

The development of efficient numerical methods for solving partial
differential equations (PDEs) on general meshes is an essential
aspect of modern scientific computing and has recently drawn considerable
attention. In this context, we mention several approaches
supporting polytopal elements and arbitrary approximation orders, such as discontinuous Galerkin (DG)
methods~\cite{BrezziManziniMariniPietraEtAl2000,Arnold2002,AntoniettiSISC2014}, hybridizable DG (HDG) methods~\cite{UnifHdgEllip2009,CockburnDongGuzman2008,Cockburn2023},
hybrid high-order (HHO) methods~\cite{DiPietroErnLemaire2014}, virtual
element methods (VEM)~\cite{BeiraodaVeigaBrezziCangianiManziniEtAl2013,BeiraodaVeigaBrezziMariniRusso2014}, and weak Galerkin (WG) methods
\cite{WangYe2013,WangYe2014}. More recently, the staggered DG (SDG) approach has attracted
attention due to its flexibility in handling complex
geometries, such as polygonal meshes, and its ability to
maintain conservation properties at the discrete level.


The DG method, introduced in the late 20th century, has evolved as
a powerful tool for solving PDEs in various fields, including
fluid dynamics, electromagnetics, and structural
mechanics~\cite{Cockburn2000, Arnold2002}. However, the classical
DG methods often require stabilization techniques when applied to
unstructured or polytopal grids.
To address these challenges, several modifications and extensions
of DG methods have been proposed. On the other hand, mixed finite
element (MFE) methods
\cite{raviart2006mixed,BoffiBrezziFortin2013} have been very
successful in modeling fluid flow and transport in porous media, as they
provide accurate and locally mass conservative velocities and
robustness with respect to heterogeneous, anisotropic, and
discontinuous coefficients.  A computational drawback of MFE is the need to solve an algebraic system of saddle point
type. Several techniques, such as mass lumping or special
quadrature rules, have been developed in the literature to address
this
issue~\cite{ArbogastWheelerYotov1997,MFMFYotov2006,MFMFYotov2019}.

The staggered DG method, first introduced by Chung and
Engquist~\cite{ChungEngquist06,ChungWave09}, represents a
significant advancement in the family of DG methods, and can be
viewed as a generalized version of the Raviart-Thomas MFE method.
It uses a staggered arrangement of primal and dual grids to ensure
local conservation of fluxes and enhanced stability. Moreover, the
resulting mass matrix is block-diagonal without resorting to mass
lumping or special quadratures rules. The salient features of SDG
methods make them desirable for solving various PDEs arising in
science and engineering~
\cite{Cheung15,ChungKim13,ChungDu17,ChungLee11,ChungQiu17,Du18,KimChungStokes13}.

Subsequent developments, such as~
\cite{LinaParkdiffusion18,ZhaoParkShin19,LinaParkelasticity20},
have extended the horizon of SDG methods from triangular meshes to
general polygonal meshes, enabling their applications in more
general geometries. The primal advantage of SDG methods lies in
their flexibility with respect to mesh geometry. Polytopal meshes,
consisting of arbitrary polygonal/polyhedral elements, are
particularly useful in applications involving complex geometrical
domains and multiphysics interactions~\cite{LinaEricPark20,
ZhaoChungLam20,LinaKim20,ZhaoPark20}.
In particular, the authors in~\cite{LinaKim20,LinaEricPark20},
proposed an SDG method for elliptic problems with interface
conditions, demonstrating how the staggered approach could handle
discontinuities across fractured porous media or material
interfaces. This method proved highly effective in addressing
issues related to mesh
irregularity~\cite{LinaKim20,KimZhaoPark2020}. Additionally, the
polygonal meshes outperform triangular meshes in terms of accuracy
\cite{LinaKim20,ArbogastWang2023}, which can be considered
justification for using polygonal elements. These contributions
have been instrumental in expanding the versatility of SDG
methods, making them applicable to a wider range of scientific and
engineering problems~\cite{ZhaoChungLam20,LinaChungParkIMA2023,
chung2024stabilization}.





In this work, we present a novel staggered discontinuous
Galerkin method specifically designed for solving the elliptic equation on polytopal meshes. Our contributions are as
follows:


 $\bullet$ We develop a new discretization scheme for the elliptic equation that employs a primal-dual grid approach to efficiently handle complex polytopal meshes.  As opposed to the previous approach~\cite{ChungWave09,ZhaoPark2020JSC}, the primal scalar variable is continuous in the polytopal mesh, and the flux vector variable is continuous in
    a facet-based staggered dual mesh. More precisely, the primal variable is broken $H^1$-conforming on the primal mesh and the flux variable is staggered and broken
    $H(\div)$-conforming on the dual mesh.
    By construction, the scalar variable on each polytope takes a form of single polynomial, which reduces the number of degrees of
    freedom significantly compared to~\cite{ChungWave09,ZhaoPark2020JSC}, where a composite
    type polynomials are employed based on a simplicial submesh.


$\bullet$  The new SDG  is hybridizable. By introducing Lagrange multipliers to impose the continuity of the normal fluxes on the primal faces, we can eliminate the flux element-wise and the saddle point system is thus reduced to a symmetric and positive definite system. It can be treated as a discretization of the primal formulation in terms of weak gradient. Therefore we name it primal SDG.


 $\bullet$ We provide some connections to other methods.
In particular, for the lowest order $k=0$ case
($P_0$-$P_1$ type) on triangular meshes, after the barycentric
refinement, the resulting matrix turns out to be identical to the
Crouzeix-Raviart non-conforming finite element matrix and only the
right hand side is different. The hybridized $P_0$-$P_1$ SDG can be thus thought as a generalization of linear non-conforming finite element to general polytopes. Another connection we show is that
the hybridized version of our primal SDG method can be viewed as a
stabilization-free weak Galerkin method or stabilization-free non-conforming virtual element method of the primal formulation of Poisson equation. In contrast, the original SDG is connected to HDG which belongs to discretization for the mixed formulation; see~\cite{ChungCockbuFu2014staggered}. 

$\bullet$ The standard continuous Galerkin method lacks locally conservative flux approximations, which are crucial in applications like streamline construction and coupled flow-transport simulations. Inaccurate velocity fields can lead to numerical errors, such as spurious sources and sinks. Our method preserves the local conservation inherent to the SDG framework, ensuring mass conservation on each primal polytopal cell while improving stability and accuracy.

$\bullet$ 
We provide numerical experiments to validate the performance of our method, showcasing its advantages in terms of convergence rates and computational efficiency.

The remainder of this paper is structured as follows: In Section~\ref{sec:SDG}, we describe the mathematical formulation of the Poisson problem
and introduce the new finite element spaces appropriate for the
SDG discretization. Then, we prove the inf-sup stability along
with the unisolvence of the degrees of freedom. In Section~\ref{sec:hybrid}, we
discuss the hybridization procedure and derive optimal convergence
error estimates.  In Section~\ref{sec:connection}, some connections to other methods
are investigated. Finally in the last
section, we detail the implementation of our method.
Then, several numerical results are presented, which confirms our
theoretical findings. 

Throughout this paper, we use
``$\lesssim\cdots $" to mean that ``$\leq C\cdots$", where
$C$ is a generic positive constant independent of $h$,
which may take different values at different appearances. And $A\eqsim B$ means $A\lesssim B$ and $B\lesssim A$.

\section{Staggered Discontinuous Galerkin Methods}\label{sec:SDG}
We consider the mixed form of Poisson equation
\begin{subequations}\label{equ_mixed_poisson}
\begin{align}
\label{equ_mixed_poisson1}
\sigma-\nabla u &= \boldsymbol{0}\qquad\, \text{in} \,\Omega,
\\
\label{equ_mixed_poisson2}
\nabla\cdot \sigma&=-f \quad\;\, \text{in} \,\Omega,
\\
\label{equ_mixed_poisson3}
u&=0 \qquad\,\, \text{on} \,\partial \Omega,
\end{align}
\end{subequations}
where $\Omega\subset \mathbb{R}^d$ is a Lipschitz polyhedral domain. The scheme and analysis can be easily generalized to a general elliptic equation by replacing \eqref{equ_mixed_poisson1} with $\boldsymbol K^{-1}\sigma = \nabla u$. 

\subsection{Primal and dual grids}
Let $\mathcal{K}_h$ be a shape regular polytopal partition of domain $\Omega$, see Fig.\ref{partitions} (a) for a 2D example.
Assume each element $K\in\mathcal{K}_h$ is star-shaped with a uniformly bounded chunkiness parameter.
Let $\mathcal{F}_h^K$ be the set of $(d-1)$-dimensional faces of $\mathcal{K}_h$, and interior faces $\mathring{\mathcal{F}}_h^{K}=\mathcal{F}_h^K \setminus \partial \Omega$. We shall call $\mathcal{K}_h$ the primal grid, and element in $\mathcal F_h^K$ the primal face or the primal edge in 2D. 

\begin{figure}[htbp]
\label{partitions}
\subfigure[The primal grid $\mathcal{K}_h$.]{
\begin{minipage}[t]{0.315\linewidth}
\centering
\includegraphics*[width=3.75cm]{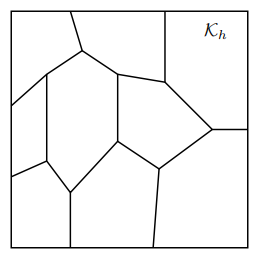} 
\quad
\end{minipage}}
\subfigure[The triangulation mesh $\mathcal{T}_h$.]
{\begin{minipage}[t]{0.315\linewidth}
\centering
\includegraphics*[width=3.75cm]{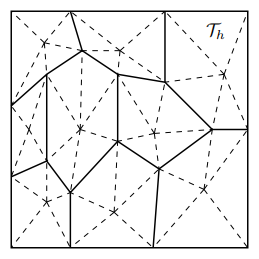} 
\end{minipage}}
\subfigure[The dual grid $\mathcal{K}_h^*$.]
{\begin{minipage}[t]{0.315\linewidth}
\centering
\includegraphics*[width=3.75cm]{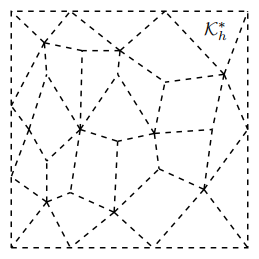} 
\end{minipage}}
\caption{The polygonal and triangle partitions.}
\end{figure}

We now introduce a dual grid. For $K\in\mathcal{K}_h$, let $\boldsymbol{x}_K$ be the center of the largest ball contained in $K$. 
Connecting $\boldsymbol{x}_K$ and the vertices of each $K\in\mathcal{K}_h$, we get $\mathcal{T}_h$, a shape regular partition of $\mathcal{K}_h$. In two dimensions, it is a triangulation; see Fig.~\ref{partitions}~(b). While $d\geq 3$, $\mathcal T_h$ may not consist of simplicies unless we assume each face $\mathcal{F}_h^K$ is a $(d-1)$-simplex. 
Restrict $\mathcal{T}_h$ to $K$ to get the partition of $K$, denoted by $\mathcal{T}_h(K)$.
Let $\mathcal{F}_h^T$ be a set of $(d-1)$-dimensional faces of $\mathcal{T}_h$ and $\mathring{\mathcal{F}}_h^{T}=\mathcal{F}_h^T \setminus \partial \Omega$.
Let $\lambda_{\boldsymbol{x}_K}$ be a piecewise linear function on $\mathcal T_h(K)$ satisfying $\lambda_{\boldsymbol{x}_K}(\boldsymbol{x}_K)=1$ and $\lambda_{\boldsymbol{x}_K}|_F=0$ for any $(d-1)$-dimensional face $F$. Such $\lambda_{\boldsymbol{x}_K}$ always exists. For a flat face $F$, we can decompose into simplicies and $\lambda_{\boldsymbol{x}_K}$ is the hat function at $\boldsymbol{x}_K$ of the triangulation $\mathcal T_h(K)$. 
Let $\nabla_h$ denote the elementwise gradient operator with respect to $\mathcal{K}_h$. 

For $F\in\mathcal{F}_h^K$, let $\omega_F$ denote the union of all polytopes in $\mathcal{T}_h$ that share the face 
$F$, forming a domain with a diamond-like shape. For $F\in \mathcal F_h^K\cap \partial \Omega$, $\omega_F$ reduces to the element in $\mathcal T_h$ containing $F$. All these $\omega_F$ form a dual grid of $\Omega$ as 
\begin{equation*}
\mathcal{K}_h^*=\{\omega_F, F\in \mathcal{F}_h^K\}.
\end{equation*}
See Fig~\ref{partitions} (c). The $(d-1)$-dimensional faces of $\partial \omega_F$ will be called the dual faces or the dual edges in 2D. In Fig~\ref{partitions} (b), the primal edge is the one with the solid line and the dual edge is drawn by the dashed line. The external boundary faces are considered as the primal faces.

We assume that the meshes $\mathcal{K}_h$ and $\mathcal{T}_h$ satisfy \cite{WeiHuangLi2021SIAM}
\begin{itemize}
\item[A1] Each element $K\in\mathcal{K}_h$ and each face $F\in\mathcal{F}_h^{K}$ is star-shaped with a uniformly bounded chunkiness parameter. 
The chunkiness parameter $\gamma_K:=\frac{h_K}{\rho_K}$, where $h_K={\rm diam}(K)$ is the diameter of $K$, $\rho_K$ is the radius of the largest ball contained in $K$.
\item[A2] There exist real numbers $\eta_K>0$, $\eta_T>0$ such that for each $K\in\mathcal{K}_h$, $h_K\leq \eta_K h_F$ for all $F\in \mathcal{F}_h^{K}(K)$, where $h_F={\rm diam}(F)$ is the diameter of $F$, and
for each  $T\in\mathcal{T}_h$, $h_T\leq \eta_T h_F$ for all $F\in \mathcal{F}_h^{T}(T)$.
\end{itemize}

We will use standard notation of Sobolev spaces and denote the $L^2$-inner product of $\Omega$ is by $(\cdot,\cdot)$ and the $L^2$-inner product on $(d-1)$-dimensional domain by $\langle \cdot, \cdot \rangle$. 

For a star-shaped domain $D$, there holds the trace inequality
$$
\|v\|_{\partial D}^2\lesssim h_D^{-1}\|v\|^2_{0,D}+h_D|v|^2_{1,D},\quad\forall v\in H^1(D).
$$
Assume the mesh $\mathcal{K}_h$ satisfies condition (A1) and $K\in\mathcal{K}_h$. It holds for any nonnegative integers $\ell$ and $i$ that
$$
\|q\|_{0,K}\lesssim h_K^{-i}\|q\|_{-i,K},\quad q\in \mathbb P_{\ell}(K),
$$
then we have
$$
\|\nabla q\|_{0,K}\lesssim h_K^{-1}\|\nabla q\|_{-1,K}\lesssim h_K^{-1}\|q\|_{0,K},\quad q\in \mathbb P_{\ell}(K).
$$

\subsection{Finite element spaces}
For $k\geq 0$ and a domain $K$, let $\mathbb P_k(K)$ denote the space of all polynomials defined on $K$ of degree less than or equal to $k$, and set $\mathbb P_k(K;\mathbb R^d)=\mathbb P_k(K)\otimes\mathbb R^d$.
For $k\geq 0$, define
$$
U_h=\{u_h\in L^2(\Omega):\; u_h|_K\in \mathbb{P}_{k+1}(K), K\in \mathcal{K}_h\}=\prod_{K\in \mathcal{K}_h} \mathbb{P}_{k+1}(K).
$$
The function $u_h\in U_h$ is a polynomial inside $K$ but discontinuous across boundary of $K$, i.e. the interior primal faces in $\mathring{\mathcal{F}}_h^{K}$. To compensate this discontinuity, we require the flux to be continuous across the primal faces but discontinuous on the dual faces.
Let
$$
\Sigma_h=\{\sigma_h\in L^2(\Omega;\mathbb{R}^d):\; \sigma_h|_T \in \mathbb P_k(T;\mathbb R^d),\; T\in \mathcal{T}_h;\; [\sigma_h\cdot \boldsymbol{n}]_F=0,\; F\in \mathring{\mathcal{F}}_h^{K}\},
$$
where $\boldsymbol n$ is a unit normal direction of $F$, and $[\cdot]$ is the jump operator defined below so that the normal component of functions in $\Sigma_h$ is continuous on the boundary of $K$. 

\begin{figure}[htbp]
\label{fig:continuity}
\subfigure[$\sigma_h\cdot\boldsymbol{n}$ is continuous on the primal edge inside the dual element.]{
\begin{minipage}[t]{0.45\linewidth}
\centering
\includegraphics*[width=3.5cm]{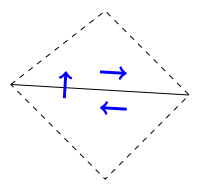}
\quad
\end{minipage}}
\qquad
\subfigure[$u_h$ is continuous across the dual edges inside the primal element. ]
{\begin{minipage}[t]{0.45\linewidth}
\centering
\includegraphics*[width=3.25cm]{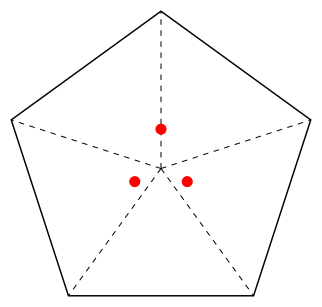}
\end{minipage}}
\caption{Different continuity of $\sigma_h$ and $u_h$.}
\end{figure}

For $F\in \mathring{\mathcal{F}}_h^{K}$, let $T_1, T_2 \in \omega_F$ satisfy $F=\partial T_1 \cap \partial T_2$.
Denote $u_h^i=u_h|_{T_i}$, $i=1,2$.
For $u_h\in U_h$, define the jump operator by
$$
[ u_h]_F=\left\{
\begin{array}{ll}
 u_h^1|_{F}- u_h^2|_{F},\; & F\in \mathring{\mathcal{F}}_h^{K},
\\
 u_h|_F,\; & F\in \partial\Omega.
\end{array}
\right.
$$
For $F\in \mathring{\mathcal{F}}_h^{T} $, let $T_1, T_2 \in \mathcal{T}_h$ satisfy $F=\partial T_1 \cap \partial T_2$.
Denote $\sigma_h^i=\sigma_h|_{T_i}$, $i=1,2$. Let $\boldsymbol{n}_i$ be the outward unit normal vector of $T_i$ on the boundary, $i=1,2$.
For $\sigma_h\in \Sigma_h$,
$$
[\sigma_h \cdot \boldsymbol{n}]_F=\left\{
\begin{array}{ll}
\sigma_h^1|_{F} \cdot \boldsymbol{n}_1+\sigma_h^2|_{F} \cdot \boldsymbol{n}_2, & F\in \mathring{\mathcal{F}}_h^{T},
\\
\sigma_h|_F \cdot \boldsymbol{n}_{\partial \Omega},& F\in \partial\Omega.
\end{array}
\right.
$$
The order of $T_1$ and $T_2$ are set by
\begin{equation*}
\sum_{T\in\mathcal{T}_h}\langle \sigma_h \cdot \boldsymbol{n}_{\partial T}, u_h \rangle_{\partial T}
=\sum_{F\in \mathring{\mathcal{F}}_h^{K}}\langle \sigma_h \cdot \boldsymbol{n}_{\partial T_1},[ u_h] \rangle_F
+\sum_{F\in \mathcal{F}_h^{T}\setminus \mathring{\mathcal{F}}_h^{K}}\langle [\sigma_h \cdot \boldsymbol{n}], u_h \rangle_F,
\end{equation*}
with $\sigma_h\in \Sigma_h$ and $u_h\in U_h$.

The following lemma is straightforward. 

\begin{lemma}
Define the following two functionals 
\begin{align*}
  &\|u_h\|_{1,h}=\left(\sum_{K\in\mathcal{K}_h}\|\nabla u_h\|_K^2
+\sum_{F\in \mathcal{F}_h^{K}}h_F^{-1}\|Q_b[u_h]\|_F^2\right)^{\frac12},
\\
  &\|\sigma_h\|_{0,h}=\left(\sum_{K\in\mathcal{K}_h}\|\sigma_h\|_K^2
+\sum_{F\in \mathcal{F}_h^{K}}h_F\|\sigma_h\cdot\boldsymbol{n}\|^2_F\right)^{\frac12},
\end{align*}
where $Q_b$ is the $L^2$-projection operator onto space $\mathbb P_{k}(F)$, on each edge $F\in \mathcal{F}_h^{K}$.
They are norms on space $U_h$ and $\Sigma_h$ respectively.
\end{lemma}
%

\begin{lemma}
We have the norm equivalences
\begin{align}
\label{normequiv1}
\|u_h\|_{1,h}^2&\eqsim\sum_{K\in\mathcal{K}_h}\|\nabla u_h\|_K^2
+\sum_{F\in \mathcal{F}_h^{K}}h_F^{-1}\|[u_h]\|_F^2\quad\forall~u_h\in U_h, \\
\label{normequiv2}
\|\sigma_h\|_{0,h}&\eqsim \|\sigma_h\| \quad\forall~\sigma_h\in \Sigma_h.
\end{align}
\end{lemma}
\begin{proof}
By the estimate of $Q_b$,
\begin{equation*}
\sum_{F\in \mathcal{F}_h^{K}}h_F^{-1}\|[u_h-Q_bu_h]\|_F^2\lesssim \sum_{K\in\mathcal K_h}\sum_{F\in\partial K}h_F^{-1}\|u_h-Q_bu_h\|_F^2\lesssim \sum_{K\in\mathcal{K}_h}\|\nabla u_h\|_K^2,
\end{equation*}
which means
\begin{align*}
\sum_{F\in \mathcal{F}_h^{K}}h_F^{-1}\|[u_h]\|_F^2\lesssim \sum_{F\in \mathcal{F}_h^{K}}h_F^{-1}\|Q_b[u_h]\|_F^2+\sum_{F\in \mathcal{F}_h^{K}}h_F^{-1}\|[u_h-Q_bu_h]\|_F^2\lesssim\|u_h\|_{1,h}^2.
\end{align*}
Hence, norm equivalence \eqref{normequiv1} is true.

Using the trace inequality and the inverse inequality we get
\begin{equation*}
  \sum_{F\in \mathcal{F}_h^{K}}h_F\|\sigma_h\cdot\boldsymbol{n}\|^2_F\lesssim \|\sigma_h\|.
\end{equation*}
This yields the norm equivalence \eqref{normequiv2}.
\end{proof}

Compared with the existing SDG methods, the primal and dual grids given in this paper switched, providing the advantage of $u$ being a piecewise polynomial on the polytopal mesh.
See Fig. \ref{priamdualk0} for $k=0$ and Fig. \ref{priamdualk1} for $k=1$ for 2-D examples. 

\begin{figure}[htbp]
\label{priamdualk0}
\subfigure[Space $\Sigma_h$ with $k=0$.]{
\begin{minipage}[t]{0.45\linewidth}
\centering
\includegraphics*[width=3.25cm]{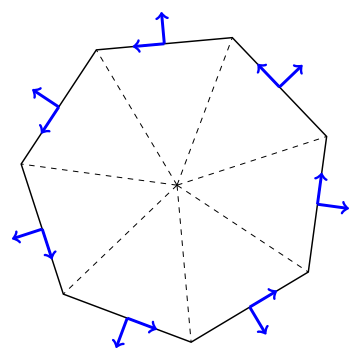}
\end{minipage}}
\subfigure[Space $U_h$ with $k = 0$.]
{\begin{minipage}[t]{0.45\linewidth}
\centering
\includegraphics*[width=3.03cm]{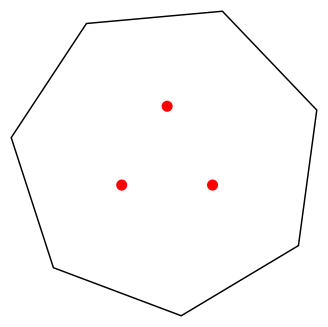}
\end{minipage}}
\caption{Degrees of freedom for $k=0$ on polygons.}
\end{figure}


\begin{figure}[htbp]
\label{priamdualk1}
\subfigure[Space $\Sigma_h$ with $k=1$.]{
\begin{minipage}[t]{0.45\linewidth}
\centering
\includegraphics*[width=3.35cm]{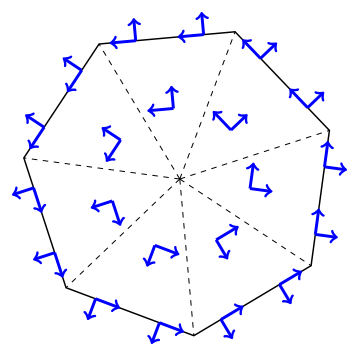}
\end{minipage}}
\subfigure[Space $U_h$ with $k = 1$.]
{\begin{minipage}[t]{0.45\linewidth}
\centering
\includegraphics*[width=3.1cm]{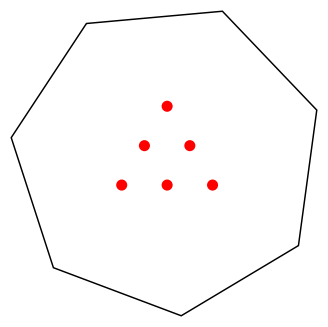}
\end{minipage}}
\caption{Degrees of freedom for $k=1$ on polygons.}
\end{figure}



\subsection{Staggered discontinuous Galerkin method}
The variational problem is given by: Find $\sigma_h\in \Sigma_h$ and $u_h\in U_h$ such that
\begin{subequations}\label{equ_variational_proble}
\begin{align}
\label{equ_variational_proble1}
a_h(\sigma_h,\tau_h)+b_h(\tau_h,u_h)&=0,\qquad\quad\;\forall~\tau_h\in \Sigma_h,
\\
\label{equ_variational_proble2}
b_h(\sigma_h,v_h)&=-(f,v_h),\;\forall~v_h\in U_h,
\end{align}
\end{subequations}
where
\begin{align*}
a_h(\sigma_h,\tau_h)&:=(\sigma_h,\tau_h), \\
b_h(\sigma_h,u_h)&:=-\sum_{K\in\mathcal{K}_h}(\sigma_h,\nabla u_h)_K+\sum_{F\in\mathcal{F}_h^{K}}\langle \sigma_h\cdot\boldsymbol{n},Q_b[u_h]\rangle_F.
\end{align*}
The Dirichlet boundary condition $u|_{\partial \Omega} = 0$ is imposed weakly  in the second term of $b_h(\sigma_h,u_h)$, i.e. $\langle \sigma_h\cdot\boldsymbol{n}, u_h\rangle_{\partial\Omega}$.
The bilinear form $b_h: \; \Sigma_h\times U_h\rightarrow \mathbb{R}$ is given based on the mixed form of Poisson equation (\ref{equ_mixed_poisson}) by testing $v_h\in U_h$ and $\tau_h\in \Sigma_h$.
The dual discretization of $b_h(\sigma_h,u_h)$ would be $(\div \sigma_h, u_h)$ but $\sigma_h$ is discontinuous inside $K$. As $u_h\in H^1(K)$, we switch the differentiation and apply the gradient to $u_h$. On the other side, the primal discretization of $b_h(\sigma_h,u_h)$ would be $-(\sigma_h, \nabla u_h)$.
As $u_h$ is discontinuous across faces in $\mathcal F_h^K$, the term involving jump $[u_h]$, which is a discrete gradient for discontinuous function, is included into $b_h(\sigma_h,u_h)$.

Using the Cauchy-Schwarz inequality, we get the continuity. 
\begin{lemma}\label{lemma_conti_ab}
The bilinear forms $a_h(\cdot,\cdot)$ and $b_h(\cdot,\cdot)$ are continuous, i.e.
\begin{align*}
a_h(\sigma_h,\tau_h) &\leq \|\sigma_h\|_{0,h}\|\tau_h\|_{0,h}, \quad\quad\forall~\sigma_h,\tau_h\in \Sigma_h,
\\
b_h(\sigma_h,u_h) &\leq \|\sigma_h\|_{0,h}\|u_h\|_{1,h}, \quad\forall~\sigma_h\in \Sigma_h, u_h\in U_h.
\end{align*}
\end{lemma}


Using the norm equivalence \eqref{normequiv2}, we get the coercivity. 
\begin{lemma}
The bilinear form $a_h(\cdot,\cdot)$ is coercive, i.e. there is a positive constant $\alpha$ independent of $h$ such that
$$
a_h(\tau_h,\tau_h)\geq \alpha \|\tau_h\|_{0,h}^2,\;\forall~\tau_h\in \Sigma_h.
$$
\end{lemma}
%

The difficulty is to prove the following inf-sup condition. We first present the result and use it to establish the well-posedness of the problem and then discuss the proof in the next subsection. 
\begin{lemma}\label{lm:infsup}
We have the inf-sup condition: there exists a positive constant $\beta$ independent of $h$ such that
$$
\inf_{u_h\in U_h}\sup_{\sigma_h\in\Sigma_h}\frac{b_h(\sigma_h,u_h)}{\|\sigma_h\|_{0,h}\|u_h\|_{1,h}}\geq \beta.
$$
\end{lemma}

\begin{theorem}
The variational problem \eqref{equ_variational_proble} is well-posed and stable. There is a unique solution pair $\sigma_h\in \Sigma_h$, $u_h\in U_h$, satisfying
\begin{equation}
\|\sigma_h\|_{0,h}+\|u_h\|_{1,h}\lesssim \|f\|.
\end{equation}
\end{theorem}
\begin{proof}
The conclusion is given by the continuity of bilinear forms $a_h(\cdot,\cdot)$ and $b_h(\cdot,\cdot)$, the inf-sup condition of $b_h(\cdot,\cdot)$  and the coercivity of $a_h(\cdot,\cdot)$. Then by the Babu\v ska-Brezzi condition~\cite{BoffiBrezziFortin2013}, we obtain the stability result. 
\end{proof}

\begin{remark}\label{rm:conservation}\rm 
By taking $v_h =\chi_{K}$ in \eqref{equ_variational_proble2}, we obtain the local conservation
$$
\int_{\partial K} \sigma_h \cdot \boldsymbol n \dd S= - \int_K f \dd x. 
$$
The standard continuous Galerkin method lacks locally conservative flux approximations, which are crucial in applications like streamline construction and coupled flow-transport simulations. 
\end{remark}

\subsection{Inf-sup condition}
This subsection is devoted to the proof of the inf-sup condition. The key is a set of degrees of freedom (DoFs) for the space $\Sigma_h$. We begin with the following DoF for $\mathbb{P}_k(T)$.

\begin{lemma}\label{lem:unisolPk}
Let $T\in\mathcal{T}_h(K)$ and $F=\partial T\cap\partial K$. For integer $k\geq0$, the degrees of freedom 
\begin{subequations}\label{Pkdof}
\begin{align}
\label{Pkdof1}
\int_Fv\,q\dd S, & \quad q\in \mathbb{P}_k(F),\\
\label{Pkdof2}
\int_Tv\,q\dx, & \quad q\in \mathbb{P}_{k-1}(T)
\end{align}
\end{subequations}
are unisolvent for space $\mathbb{P}_k(T)$.
\end{lemma}
\begin{proof}
Since 
\begin{equation*}
\dim \mathbb{P}_k(F) +\dim \mathbb{P}_{k-1}(T)=\dim\mathbb{P}_k(T),
\end{equation*}
the number of DoFs \eqref{Pkdof} equals the dimension of space $\mathbb{P}_k(T)$.
Assume $v\in\mathbb{P}_k(T)$ satisfying all the DoFs \eqref{Pkdof} vanishes. The vanishing DoF \eqref{Pkdof1} means $v|_F=0$. 
It is evident that $v\in\mathbb{P}_k(T)$ can be represented by the Bernstein basis functions, subject to the condition that $v|_F=0$, implying that $v=\lambda_{\boldsymbol{x}_K}q$, $q\in\mathbb{P}_{k-1}(T)$.
 Apply the vanishing DoF \eqref{Pkdof2} to acquire $v=0$.
\end{proof}

For $k\geq 0$, introduce a discontinuous vector space
\begin{equation*}
\Sigma_h^{-1}(K)=\{\tau\in L^2(K;\mathbb{R}^d): \tau|_{T}\in\mathbb P_k(T;\mathbb{R}^d), T\in \mathcal{T}_h(K)\}=\prod_{T\in\mathcal{T}_h(K)}\mathbb{P}_{k}(T;\mathbb{R}^d).
\end{equation*}
Define a bubble function space
\begin{equation*}
\mathbb{B}_k(K;\mathbb{R}^d)=\left\{\tau\in \Sigma_h^{-1}(K): 
\tau \cdot \boldsymbol{n}|_{\partial K}=0\right\}.
\end{equation*}

Let $K\in\mathcal{K}_h$ be a polytope with the number of the $(d-1)$-dimensional boundary faces being $n$. 
Denote by $\boldsymbol{n}^F$ the unit outward normal vector of $K$ on boundary face $F\in\partial K$, and let $\boldsymbol{t}_i^{F}$, $i=1,...,d-1$ be the unit tangent vectors on the boundary face $F\in \partial K$. Introduce the characteristic function $\chi_{T}$
$$
\chi_{T}=
\left\{\begin{array}{ll}
1,& T,
\\
0,& \mathcal{T}_h(K)\setminus T.
\end{array}
\right.
$$

\begin{lemma}\label{lemma_DOF_discreption}
For $K\in\mathcal{K}_h$ and integer $k\geq0$,
we have the geometric decomposition
\begin{equation}\label{bubblegeodecomp}
\mathbb{B}_k(K;\mathbb{R}^d)=\Oplus_{i=1}^n\big(\lambda_{\boldsymbol{x}_K}\boldsymbol{n}^{F_i}\chi_{T_i}\mathbb{P}_{k-1}(T_i) \oplus \Oplus_{j=1}^{d-1}\boldsymbol{t}_j^{F_i}\chi_{T_i}\mathbb{P}_{k}(T_i)\big).
\end{equation}
\end{lemma}
\begin{proof}
It can be verified directly that
\begin{equation*}
\Oplus_{i=1}^n\big(\lambda_{\boldsymbol{x}_K}\boldsymbol{n}^{F_i}\chi_{T_i}\mathbb{P}_{k-1}(T_i) \oplus \Oplus_{j=1}^{d-1}\boldsymbol{t}_j^{F_i}\chi_{T_i}\mathbb{P}_{k}(T_i)\big) \subseteq \mathbb{B}_k(K;\mathbb{R}^d).
\end{equation*}
On the other side, let $\tau\in\mathbb{B}_k(K;\mathbb{R}^d)$.
For $i=1,\ldots, n$, $\tau|_{T_i}\in\mathbb{P}_{k}(T_i;\mathbb R^d)$. 
We use frame $\{ \boldsymbol{n}^{F_i}, \boldsymbol{t}_j^{F_i}, j=1,\ldots,d-1 \}$ to expand the vector function in $T_i$.
By $\tau\cdot\boldsymbol{n}|_{F_i}=0$, we have $\tau|_{T_i}\in \lambda_{\boldsymbol{x}_K}\boldsymbol{n}^{F_i}\mathbb{P}_{k-1}(T_i) \oplus \Oplus_{j=1}^{d-1}\boldsymbol{t}_j^{F_i}\mathbb{P}_{k}(T_i)$.
Therefore, the geometric decomposition \eqref{bubblegeodecomp} holds.
\end{proof}

\begin{lemma}\label{lem:unisolSigmak}
For $K\in\mathcal{K}_h$ and integer $k\geq0$, the degrees of freedom 
\begin{subequations}\label{equ_dofs_old}
\begin{align}
\int_{F}\sigma \cdot \boldsymbol{n}\, q \dd S,&\quad q\in \mathbb{P}_k(F), F\in \partial K ,\label{equ_dofs_old1} \\
\int_K\sigma \cdot \tau \dx,& \quad  \tau \in \mathbb{B}_k(K;\mathbb{R}^d), \label{equ_dofs_old2}
\end{align}
\end{subequations}
are unisolvent for space $\Sigma_h^{-1}(K)$.
\end{lemma}
\begin{proof}
Thanks to Lemma~\ref{lem:unisolPk} and the geometric decomposition \eqref{bubblegeodecomp}, the number of DoFs \eqref{equ_dofs_old} equals $\dim\Sigma_h^{-1}(K)$. Vanishing \eqref{equ_dofs_old1} implies $\sigma \in \mathbb{B}_k(K;\mathbb{R}^d)$ and then vanishing \eqref{equ_dofs_old2} implies $\sigma = \boldsymbol{0}$. 
This ends the proof.
\end{proof}

We move some local tangential bubble functions in~(\ref{equ_dofs_old2}) to form the $\mathbb{P}_k$ polynomials on $K$ in (\ref{equ_dofs_dd_polytope_new3}). Let 
\begin{equation}
  \displaystyle
\left\{\boldsymbol{t}_\ell, \ell=1,...,d\right\}\subset
\left\{\boldsymbol{t}_j^{F_i}, i=1,...,n, j=1,...,d-1\right\},\label{equ_edge_vector_property}
\end{equation}
be a linearly independent set, i.e.
\begin{equation}
  \displaystyle
\mathbb{R}^d=\text{span}\left\{\boldsymbol{t}_\ell, \ell=1,...,d\right\}\label{equ_assumption_ddpolytope}.
\end{equation}
Such a set of $\left\{\boldsymbol{t}_\ell, \ell=1,...,d\right\}$ must exist, otherwise,
$K$ would be contained in a hypersurface of dimension less than $d$, since all the faces of 
$K$ would be contained in this lower-dimensional hypersurface. 

\begin{figure}[htbp]
\label{fig_bubblefunction}
\subfigure[DoFs given in (\ref{equ_dofs_old}).]{
\begin{minipage}[t]{0.45\linewidth}
\centering
\includegraphics*[width=4cm]{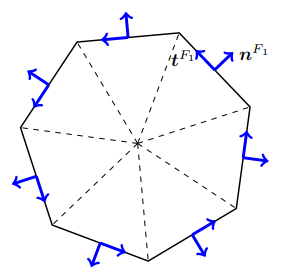}
\end{minipage}}
\subfigure[DoFs given in (\ref{equ_dofs_dd_polytope_new}).]
{\begin{minipage}[t]{0.45\linewidth}
\centering
\includegraphics*[width=4cm]{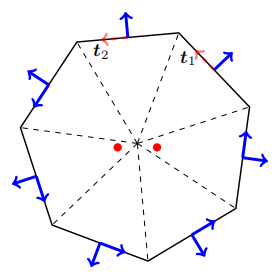}
\end{minipage}}
\caption{ Explanation of DoFs \eqref{equ_dofs_old} and \eqref{equ_dofs_dd_polytope_new}. The interior DoFs replace two tangent vectors (in dashed red).}
\end{figure}

\begin{lemma}
Define new DoFs for $\Sigma_h^{-1}(K)$ by
\begin{subequations}\label{equ_dofs_dd_polytope_new}
\begin{eqnarray}
&&\int_{F}\sigma \cdot \boldsymbol{n}\,q\dd S,\quad q \in \mathbb{P}_k(F), F\in \partial K ,\label{equ_dofs_dd_polytope_new1}
\\[0.1in]
&&\int_K\sigma \cdot \tau \dx, \quad \tau\in \mathbb{B}_k(K;\mathbb{R}^d)\setminus \Oplus_{\ell=1}^{d}\mathbb{P}_k(T_{\boldsymbol{t}_\ell})\boldsymbol{t}_\ell \chi_{T_{\boldsymbol{t}_\ell}},\label{equ_dofs_dd_polytope_new2}
\\[0.1in]
&&\int_K\sigma\cdot\tau \dx, \quad \tau\in \mathbb{P}_k(K;\mathbb{R}^d),\label{equ_dofs_dd_polytope_new3}
\end{eqnarray}
\end{subequations}
where $T_{\boldsymbol{t}_\ell}\in\mathcal{T}_h(K)$ is the subelement with tangent vector $\boldsymbol{t}_\ell$ for $ \ell=1,...,d$, as in \eqref{equ_edge_vector_property}. DoFs \eqref{equ_dofs_dd_polytope_new} are unisolvent for $\Sigma_h^{-1}(K)$.
\end{lemma}
\begin{proof}
By comparing DoFs \eqref{equ_dofs_dd_polytope_new} with DoFs \eqref{equ_dofs_old}, we have from \eqref{equ_assumption_ddpolytope} that the number of DoFs (\ref{equ_dofs_dd_polytope_new}) equals to the dimension of $\Sigma_h^{-1}(K)$.



Assume that $\sigma\in \Sigma_h^{-1}(K)$ satisfying DoFs (\ref{equ_dofs_dd_polytope_new}) vanish.
The vanishing DoFs \eqref{equ_dofs_dd_polytope_new1}-\eqref{equ_dofs_dd_polytope_new2} imply 

$$
\sigma=\sum_{\ell=1}^{d} q_{\ell} \boldsymbol{t}_\ell  \chi_{T_{\boldsymbol{t}_\ell}}\quad q_{\ell}\in\mathbb{P}_k(T_{\boldsymbol{t}_\ell}).
$$
The domain of polynomial $q_{\ell}$ can be extended to $K$.
Let $\{\hat{\boldsymbol{t}}_\ell, \ell=1, \ldots, d\}$ be the dual bases of $\{\boldsymbol{t}_\ell, \ell=1, \ldots ,d\}$, i.e. $(\hat{\boldsymbol{t}}_i, \boldsymbol t_j) = \delta_{ij}$ for $i,j=1,\ldots, d$.
Taking $\tau=q_{\ell}\hat{\boldsymbol{t}}_\ell$, $\ell=1,...,d$ in
the vanishing DoFs (\ref{equ_dofs_dd_polytope_new3}), we get $\sigma=\boldsymbol{0}$, which finishes the proof.
\end{proof}

Notice that DoFs \eqref{equ_dofs_dd_polytope_new} is introduced to verify the inf-sup condition. In implementation, we can use basis dual to DoFs \eqref{equ_dofs_old} or use the hybridization discussed in the next section.


\begin{proof}[Proof of Lemma \ref{lm:infsup}]
For any $u_h\in U_h$, let $\sigma_h\in\Sigma_h$ satisfy
\begin{equation}\label{equ_dof_infsup}
\begin{aligned}
\quad\,(\sigma_h,\tau_h)_K &=-(\nabla u_h,\tau_h)_K, \;\quad\;\, \forall~\tau_h\in\mathbb{P}_k(K;\mathbb{R}^d), K\in\mathcal{K}_h,
\\
\langle \sigma_h\cdot\boldsymbol{n}, w\rangle_F &=h_F^{-1}\langle Q_b[u_h], w \rangle_F,\; \forall~w\in\mathbb{P}_k(F), F\in \mathcal{F}_h^{K},
\end{aligned}
\end{equation}
and the other DoFs in (\ref{equ_dofs_dd_polytope_new2}) vanish. It follows from the scaling argument that
\begin{eqnarray*}
\|\sigma_h\|_{0,h}^2\lesssim \left\{\sum_{T\in\mathcal{T}_h}\|\nabla u_h\|_T^2
+\sum_{F\in \mathcal{F}_h^{K}}h_F^{-1}\|Q_b[u_h]\|_{F}^2\right\}=\|u_h\|_{1,h}^2.
\end{eqnarray*}
Using (\ref{equ_dof_infsup}), we have
\begin{align*}
b_h(\sigma_h,u_h)&= \sum_{K\in\mathcal{K}_h}-(\sigma_h,\nabla u_h)_K+\sum_{F\in\mathcal{F}_h^{K}}\langle \sigma_h\cdot\boldsymbol{n},Q_b[u_h]\rangle_F
\\
&=\sum_{K\in\mathcal{K}_h}\|\nabla u_h\|_K^2
+\sum_{F\in \mathcal{F}_h^{K}}h_F^{-1}\|Q_b[u_h]\|_{F}^2
 = \|u_h\|_{1,h}^2.
\end{align*}
Then we get the inf-sup condition.
\end{proof}
We will give optimal order error analysis after we present a hybridization of our SDG method in the next section.

\section{Hybridization}\label{sec:hybrid}
In this section, we consider the hybridization of the staggered discontinuous Galerkin method \eqref{equ_variational_proble}, and present the optimal order error analysis. 

\subsection{Spaces and weak differential operators}
In space $\Sigma_h$, the normal flux is continuous on the primal faces. Following the classical approach in~\cite{ArnoldBrezzi1985nonconforming}, such continuity can be weakly imposed by introducing Lagrange multipliers $u_b$ on the primal faces. Let
$$
M_h=\{u_h=\{u_0,u_b\}: \; u_0|_{K} \in \mathbb{P}_{k+1}(K), u_b|_F\in \mathbb{P}_k(F), \; K\in\mathcal{K}_h, \; F\in \mathcal{F}_h^{K}\},
$$
and
$$
\Sigma^{-1}_h=\{\sigma_h\in L^2(\Omega;\mathbb{R}^d):\; \sigma_h|_T\in \mathbb{P}_k(T;\mathbb{R}^d),\, T\in \mathcal{T}_h\}.
$$
Let $M_h^0$ be a subspace of $M_h$ with vanishing boundary values on $\partial\Omega$, i.e.
$$
M_h^0=\{u_h=\{u_0,u_b\}\in M_h: \; u_b|_F= 0, F\in\partial\Omega\}.
$$
For $u_h, v_h \in M_h$, define
$$
(u_h,v_h)_{0,h}=\sum_{K\in\mathcal{K}_h}(u_0,v_0)_K+\sum_{F\in \mathcal{F}_h^{K}}h_F\langle u_b,v_b\rangle_{F},
$$
which induces an $L^2$-type norm $\|u_h\|_{0,h}=(u_h,u_h)_{0,h}^{1/2}.$

Define the weak gradient operator $$\nabla_w: M_h\rightarrow \Sigma_h^{-1}$$ as
\begin{equation}\label{eq:wg}
\begin{aligned}
&(\nabla_w u_h)|_K =\nabla_{w,K}u_h \in \Sigma_h^{-1}(K), \quad K\in\mathcal{K}_h,
\\
  \displaystyle
&(\nabla_{w,K} u_h,\tau)_K =(\nabla u_0, \tau)_K+\langle u_b-u_0,\tau\cdot\boldsymbol{n}\rangle_{\partial K},\quad \forall~\tau\in\Sigma_h^{-1}(K).
\end{aligned}
\end{equation}

Define the weak divergence operator by
$$
\text{div}_w: \Sigma^{-1}_h \rightarrow M_h,
$$
by
\begin{equation}\label{eq:divw}
\text{div}_w \sigma_h=\{\text{div}_{w,K}\sigma_h, -h_F^{-1}[\sigma_h\cdot\boldsymbol{n}]\}_{\mathcal{K}_h,\mathcal{F}_h^K},
\end{equation}
where the local weak divergence $\text{div}_{w,K} \sigma_h \in \mathbb{P}_{k+1}(K)$ is defined as
$$
(\text{div}_{w,K}\sigma_h,u_0)_K=
\sum_{T\in\mathcal{T}_h(K)}(\text{div} \sigma_h, u_0)_T 
-\sum_{F\in \mathcal{F}_h^T(K)\setminus \partial K} \langle  [\sigma_h\cdot\boldsymbol{n}], u_0\rangle_{F},
$$
for any $u_0\in \mathbb{P}_{k+1}(K)$.

For $u_h\in M_h$ and $\sigma_h\in \Sigma_h^{-1}$, using the integration by parts, we have
\begin{align*}
 (\nabla_w u_h,\sigma_h)&=\sum_{K\in\mathcal{K}_h}(\nabla_{w,K} u_h,\sigma_h)_K\\
& =\sum_{K\in\mathcal{K}_h}(\nabla u_0,\sigma_h)_K+\langle u_b-u_0,\sigma_h\cdot\boldsymbol{n}\rangle_{\partial K}\\
& =\sum_{T\in\mathcal{T}_h}-(u_0,\text{div}_h \sigma_h)_T +\sum_{F\in \mathcal{F}_h^T\setminus \mathcal{F}_h^K}\langle u_0, [\sigma_h\cdot\boldsymbol{n}]\rangle_{F} 
   +\sum_{F\in \mathcal{F}_h^K} \langle u_b,[\sigma_h\cdot\boldsymbol{n}]\rangle_{F}  \\
& =-\sum_{K\in\mathcal{K}_h}(\text{div}_{w,K}\sigma_h,u_0)_K-\sum_{F\in \mathcal{F}_h^{K}}h_F\langle u_b,-h_F^{-1} [\sigma_h\cdot\boldsymbol{n}] \rangle_{F}\\
& = -(\text{div}_w\sigma_h,u_h)_{0,h}.
\end{align*} 
That gives the fact that the weak gradient $\nabla_w$ defined by \eqref{eq:wg} and the weak divergence $\div_w$ defined by \eqref{eq:divw} are adjoint operators with respect to appropriate inner products.

\subsection{Hybridized variational form}
Let $a_h: \; \Sigma^{-1}_h\times \Sigma^{-1}_h\rightarrow \mathbb{R}$ be given by
$$
a_h(\sigma_h,\tau_h):=\sum_{T\in\mathcal{T}_h}(\sigma_h,\tau_h)_T.
$$
Define $b_h: \Sigma^{-1}_h \times M_h \rightarrow \mathbb{R}$ by
\begin{align*}
b_h(\sigma_h,u_h)&=(\text{div}_w\sigma_h,u_h)_{0,h}=-(\nabla_w u_h,\sigma_h)
\\
&=\sum_{K\in\mathcal{K}_h}\left\{-(\sigma_h,\nabla u_0)_K+\langle \sigma_h\cdot\boldsymbol{n}, Q_bu_0-u_b\rangle_{\partial
K}\right\}.
\end{align*}

The hybridized variational problem is given by: Find $\sigma_h\in \Sigma^{-1}_h$ and $u_h\in M_h^0$ such that
\begin{subequations}\label{equ_hy}
\begin{align}
\label{equ_hy1}
a_h(\sigma_h,\tau_h)+b_h(\tau_h,u_h)&=0,\qquad\quad\;\forall~\tau_h\in \Sigma^{-1}_h,
\\
\label{equ_hy2}
b_h(\sigma_h,v_h)&=(-f,v_0),\;\forall~v_h\in M_h^0.
\end{align}
\end{subequations}

To address the continuity, with abuse of notation, we introduce
$$
\|u_h\|_{1,h}^2:=\sum_{K\in\mathcal{K}_h}\|\nabla u_0\|_K^2+\sum_{K\in\mathcal{K}_h}h_K^{-1}\|u_0-u_b\|_{\partial K}^2, \quad u_h\in M_h,
$$
and
$$
\|\sigma_h\|_{0,h}^2:=\sum_{K\in\mathcal{K}_h}\|\sigma_h\|_K^2
+\sum_{K\in \mathcal{K}_h}h_K\|\sigma_h \cdot \boldsymbol{n}\|^2_{\partial K},\quad \sigma_h\in \Sigma_h^{-1}.
$$
It is easy to show that $\|\cdot\|_{1,h}$ is a norm on space $M_h^0$
and $\|\cdot\|_{0,h}$ is a norm on space $\Sigma_h^{-1}$, respectively. The norm $\|\cdot\|_{1,h}$ can be applied to functions $u\in H^1(\Omega)$, in which case, $u$ is piecewise smooth on $\mathcal T_h$ and by taking $u_0 = u$ and $u_b = {\rm tr}_F u$  for all $F\in \mathcal F^K_h$, we have $\|u\|_{1,h} = \|\nabla u\|$. Similarly $\|\cdot\|_{0,h}$ can be defined for function $\sigma\in L^2(\Omega)$ and $\sigma\cdot \boldsymbol n|_F\in L^2(F)$ for all $F\in \mathcal F^K_h$.

Following the argument for proving \eqref{normequiv1}, we have the norm equivalence
$$
\|u_h\|_{1,h}^2\eqsim \sum_{K\in\mathcal{K}_h}\|\nabla u_0\|_K^2+\sum_{K\in\mathcal{K}_h}h_K^{-1}\|Q_bu_0-u_b\|_{\partial K}^2, \quad u_h\in M_h.
$$

Again using the Cauchy-Schwarz inequality, we have the continuity. 

\begin{lemma}
The bilinear forms $a_h(\cdot,\cdot)$ and $b_h(\cdot,\cdot)$ are continuous. 
\begin{equation*}
\begin{aligned}
a_h(\sigma_h,\tau_h) &\leq \|\sigma_h\|_{0,h}\|\tau_h\|_{0,h}, \quad\,\forall~\sigma_h,\tau_h\in \Sigma_h^{-1},
\\
b_h(\sigma_h,u_h) &\lesssim\|\sigma_h\|_{0,h}\|u_h\|_{1,h}, \quad\forall~\sigma_h\in \Sigma_h^{-1}, u_h\in M_h^0.
\end{aligned}
\end{equation*}
\end{lemma}

Using the norm equivalence \eqref{normequiv2}, we get the coercivity. 
\begin{lemma}The bilinear form $a_h(\cdot,\cdot)$ is coercive, i.e. there is a positive constant $\alpha$ such that
$$
a_h(\tau_h,\tau_h)\geq \alpha \|\tau_h\|_{0,h}^2,\quad \forall~\tau_h\in \Sigma_h^{-1}.
$$
\end{lemma}

The key is to verify the inf-sup condition.
\begin{lemma}There exists a positive constant $\beta$ independent of $h$ such that
\begin{equation}\label{eqn_inf_sup_WG}
\inf_{u_h\in M_h^0}\sup_{\sigma_h\in\Sigma^{-1}_h}\frac{b_h(\sigma_h,u_h)}{\|\sigma_h\|_{0,h}\|u_h\|_{1,h}}\geq\beta>0.
\end{equation}
\end{lemma}
\begin{proof}
For any $u_h\in M_h^0$, let $\sigma_h\in\Sigma_h^{-1}$ satisfy that on each $K\in\mathcal{K}_h$,
\begin{align}
\label{equ_dof_infsup_hy1}
(\sigma_h,\tau_h)_K &= -(\nabla u_0,\tau_h)_K, \qquad\quad\; \forall~\tau_h\in\mathbb{P}_k(K;\mathbb{R}^d), 
\\
\label{equ_dof_infsup_hy2}
\langle \sigma_h\cdot\boldsymbol{n}, w\rangle_F &= h_K^{-1} \langle Q_bu_0-u_b, w \rangle_F,\; \forall~w\in\mathbb{P}_k(F), F\in \partial K,
\end{align}
and the other DoFs in (\ref{equ_dofs_dd_polytope_new2}) vanish. It follows from the scaling argument that
\begin{eqnarray*}
\|\sigma_h\|_{0,h}^2\lesssim\sum_{K\in\mathcal{K}_h}(\|\nabla u_0\|_K^2
+h_K^{-1}\|Q_bu_0-u_b\|_{\partial K}^2)\eqsim\|u_h\|_{1,h}^2.
\end{eqnarray*}
From \eqref{equ_dof_infsup_hy1}-\eqref{equ_dof_infsup_hy2}, we have
\begin{align*}
b_h(\sigma_h,u_h)&=\sum_{K\in\mathcal{K}_h}\left\{-(\sigma_h,\nabla u_0)_K+\langle \sigma_h\cdot\boldsymbol{n},Q_bu_0-u_b\rangle_{\partial
K}\right\}
\\
&=\sum_{K\in\mathcal{K}_h}\left\{\|\nabla u_0\|_K^2
+h_K^{-1}\|Q_bu_0-u_b\|_{\partial K}^2\right\}\eqsim\|u_h\|_{1,h}^2.
\end{align*}
We finished the proof.
\end{proof}

Therefore the variational problem (\ref{equ_hy}) is well-posed and stable. Next we show the equivalence to the original formulation.

\begin{theorem}
The hybridized variational form \eqref{equ_hy} has a unique solution $\sigma_h\in \Sigma^{-1}_h$ and $u_h=(u_0, u_b)\in M_h^0$ for $k\geq 0$, and 
\begin{equation}\label{eq:hystability}
\|\sigma_h\|_{0,h}+\|u_h\|_{1,h} \lesssim \| f\|. 
\end{equation}
Moreover, $\sigma_h\in \Sigma_h$ and $u_0\in U_h$ satisfy the variational form \eqref{equ_variational_proble}.
\end{theorem}
\begin{proof}
The discrete method~\eqref{equ_hy} is well-posed thanks to the discrete inf-sup condition
\eqref{eqn_inf_sup_WG}. 
The stability~\eqref{eq:hystability} is from the Babu\v ska-Brezzi theory. 


By taking $v_0=0$ in \eqref{equ_hy2}, we get $\sum_{K\in\mathcal{K}_h} \langle \sigma_h\cdot\boldsymbol{n}, v_b\rangle_{\partial K}=0$. That is $\sigma_h\in\Sigma_h$. 
Then $\sigma_h\in \Sigma_h$ satisfies the variational form \eqref{equ_variational_proble2}.
On the other hand, 
restricting $\tau_h\in\Sigma_h$ in \eqref{equ_hy1}, we know that $\sigma_h\in \Sigma_h$ and $u_0\in U_h$ satisfy the variational form~\eqref{equ_variational_proble1}.

\end{proof}
In the implementation, as $\sigma_h$ is piecewise polynomial, it can be eliminated element-wise and the saddle point problem \eqref{equ_hy} can be reduced to a symmetric and positive definite system; see Section \ref{sec:WG} for detailed discussion.

\subsection{Error analysis}
For any $\sigma\in H^{k+1}(\Omega; \mathbb{R}^d)$, define the interpolation $\sigma_I\in \Sigma_h$ by the DoF
\begin{subequations}\label{eq:interpolant}
\begin{align}
\int_{F}\sigma_I \cdot \boldsymbol{n}\,q\dd S&=\int_{F}\sigma \cdot \boldsymbol{n}\,q\dd S,\quad q \in \mathbb{P}_k(F), F\in \mathcal{F}_h^K ,
\\
\int_K\sigma_I \cdot \tau \dx&=\int_K\sigma \cdot \tau \dx, \quad\; \tau\in \mathbb{B}_k(K;\mathbb{R}^d)\setminus \Oplus_{\ell=1}^{d}\mathbb{P}_k(T_{\boldsymbol{t}_\ell})\boldsymbol{t}_\ell \chi_{T_{\boldsymbol{t}_\ell}}, K\in\mathcal{K}_h,
\\
\label{eq:interpolantdofPk} \int_K\sigma_I\cdot\tau \dx&=\int_K\sigma\cdot\tau \dx, \quad\; \tau\in \mathbb{P}_k(K;\mathbb{R}^d), K\in\mathcal{K}_h.
\end{align}
\end{subequations}
By the standard interpolation error estimate,
\begin{equation}\label{eqn_est_sigma_I}
\|\sigma-\sigma_I\|_{0,h}\lesssim h^{k+1}\|\sigma\|_{k+1}.
\end{equation}

\begin{theorem}
  Let $u\in H^{k+2}(\Omega)$ and $\sigma\in H^{k+1}(\Omega; \mathbb{R}^d)$ be the solution of \eqref{equ_mixed_poisson}, $k\geq 0$. 
  Let $u_h\in M_h^0$ and $\sigma_h\in\Sigma_h$ be the solution of  \eqref{equ_hy}.
  We have the following error estimates
  \begin{equation}\label{eq:energyerror}
    \|\sigma-\sigma_h\|_{0,h}+\|u-u_h\|_{1,h}\lesssim h^{k+1}(\|\sigma\|_{k+1}+\|u\|_{k+2}).
  \end{equation}
\end{theorem}
\begin{proof}
Let $Q_h=\{Q_0,Q_b\}: L^{2}(\Omega)\times \Pi_{F\in\mathcal{F}_h}L^2(F)\rightarrow M_h$ be the $L^2$-projection operator. We first derive an error equation on the difference $\sigma_h-\sigma_I$ and $u_h-Q_hu$. By \eqref{equ_hy},
\begin{equation}\label{errores1}
\begin{aligned}
 &a_h(\sigma_h-\sigma_I, \tau_h)+b_h(\tau_h, u_h-Q_hu)
 \\
&=a_h(\sigma_h,\tau_h)+b_h(\tau_h,u_h)-a_h(\sigma_I,\tau_h)-b_h(\tau_h,Q_hu)
 \\
&=-(\sigma_I,\tau_h)+(\tau_h, \nabla_w(Q_hu)),
\end{aligned}
\end{equation}
and 
\begin{equation}\label{errores2}
    b_h(\sigma_h-\sigma_I,v_h)= -(f,v_0)+(\sigma_I, \nabla_wv_h).
\end{equation}
Now the right hand side involves $\sigma_I$ and $Q_hu$ but not $\sigma_h$ and $u_h$.

Using \eqref{equ_mixed_poisson1}, the Cauchy-Schwarz inequality, the projection inequality, and (\ref{eqn_est_sigma_I}), we have
\begin{align*}
&-(\sigma_I,\tau_h)+(\tau_h, \nabla_w(Q_hu))\\
={}& (\sigma,\tau_h)-(\nabla u,\tau_h)  - (\sigma_I,\tau_h)
  +\sum_{K\in\mathcal{K}_h}\left\{(\tau_h,\nabla Q_0 u)_K-\langle \tau_h\cdot\boldsymbol{n},Q_0u-Q_bu\rangle_{\partial
K}\right\}
 \\
={}&(\sigma-\sigma_I,\tau_h)
 +\sum_{K\in\mathcal{K}_h} \left \{(\tau_h,\nabla Q_0u-\nabla u)_K-\langle \tau_h\cdot\boldsymbol{n},Q_0u-u\rangle_{\partial
K}\right \}
 \\
\leq {}& \|\sigma-\sigma_I\|\|\tau_h\|+\|\tau_h\|\Big(\sum_{K\in\mathcal K_h}\|\nabla Q_0u-\nabla u\|_K^2\Big)^{\frac12}
   \\
&  + \Big (\sum_{F\in\mathcal{F}_h^K} h_F^{-1}\|Q_0u-u\|_{F}^2  \Big )^{\frac12}
  \Big (\sum_{F\in\mathcal{F}_h^K} h_F \|\tau_h\cdot \boldsymbol{n}\|_{F}^2  \Big )^{\frac12}
     \\
 \lesssim{}&  h^{k+1}(\|\sigma\|_{k+1}+\|u\|_{k+2})\|\tau_h\|_{0,h}.
\end{align*}

  Similarly, using \eqref{equ_mixed_poisson2}, \eqref{eq:interpolantdofPk}, the Cauchy-Schwarz inequality, the projection inequality, the trace inequality, and (\ref{eqn_est_sigma_I}), we have
  \begin{equation*}
\begin{aligned}
&-(f,v_0)+(\sigma_I, \nabla_wv_h)  \\
= {}&(\nabla\cdot\sigma,v_0)+\sum_{K\in\mathcal{K}_h}\left\{(\sigma_I,\nabla v_0)_K-\langle \sigma_I\cdot\boldsymbol{n},v_0-v_b\rangle_{\partial
K}\right\}
 \\
  ={} &\sum_{K\in\mathcal{K}_h}\left\{(\sigma_I-\sigma,\nabla v_0)_K+\langle \sigma\cdot\boldsymbol{n}-\sigma_I\cdot\boldsymbol{n},v_0-v_b\rangle_{\partial
K}\right\}
 \\
\lesssim {}& \|v_h\|_{1,h}
\Big (\sum_{K\in\mathcal{K}_h} h_K\|\sigma\cdot\boldsymbol{n}-\sigma_I\cdot\boldsymbol{n}\|_{\partial K}^2  \Big )^{\frac12}
   \\
  \lesssim {}& h^{k+1}\|\sigma\|_{k+1}\|v_h\|_{1,h}.
\end{aligned}
\end{equation*}
Then we have from \eqref{eq:hystability}, \eqref{errores1}, and \eqref{errores2} that
$$
\| \sigma_I - \sigma_h \| +\|u_h-Q_hu\|_{1,h}\lesssim h^{k+1}(\|\sigma\|_{k+1}+\|u\|_{k+2}).
$$
Together with the triangular inequality and the interpolation error estimate, we finish the proof.
\end{proof}

Next, we estimate $\|u-u_0\|$. Consider the dual problem
\begin{equation}\label{eq:dualproblem}
\tilde{\sigma}=\nabla \tilde{u},\; -\div\tilde{\sigma}=u-u_0\quad \textrm{ in } \Omega ; \quad \tilde{u}|_{\partial\Omega}=0.
\end{equation}
Assume we have the regularity
\begin{equation}\label{eq:dualregularity}
\|\tilde\sigma\|_1+\|\tilde u\|_2 \lesssim \|u-u_0\|_0.
\end{equation}
The regularity \eqref{eq:dualregularity} holds when $\Omega$ is a two-dimensional convex domain \cite{Dauge1988}. 

\begin{theorem}
Let $u\in H^{k+2}(\Omega)$ and $\sigma\in H^{k+1}(\Omega; \mathbb{R}^d)$ be the solution of \eqref{equ_mixed_poisson}, $k\geq 0$. Let $u_h\in M_h^0$ and $\sigma_h\in\Sigma_h$ be the solution of \eqref{equ_hy}.
Assume the regularity~\eqref{eq:dualregularity} holds.
We have 
\begin{equation}\label{eq:uL2error}
\|u-u_0\|\lesssim h^{k+2}\|u\|_{k+2}.
\end{equation}
\end{theorem}
\begin{proof}
From \eqref{eq:dualproblem}, the integration by parts and \eqref{equ_hy1}, it follows
\begin{align*}
&\|u-u_0\|^2=-(\div\tilde{\sigma}, u-u_0)=(\tilde{\sigma}, \nabla_h(u-u_0)) -\sum_{K\in\mathcal K_h}\langle\tilde{\sigma}\cdot\boldsymbol{n}, u-u_0\rangle_{\partial K} \\
&\quad=(\tilde{\sigma}-\tilde{\sigma}_I, \nabla_h(u-u_0)) -\sum_{K\in\mathcal K_h}\langle(\tilde{\sigma}-\tilde{\sigma}_I)\cdot\boldsymbol{n}, u_b-u_0\rangle_{\partial K}  + (\sigma,\tilde{\sigma}_I) + b_h(\tilde{\sigma}_I, u_h) \\
&\quad=(\tilde{\sigma}-\tilde{\sigma}_I, \nabla_h(u-u_0)) -\sum_{K\in\mathcal K_h}\langle(\tilde{\sigma}-\tilde{\sigma}_I)\cdot\boldsymbol{n}, u_b-u_0\rangle_{\partial K}  + (\sigma-\sigma_h,\tilde{\sigma}_I).
\end{align*}
Apply \eqref{eqn_est_sigma_I} and \eqref{eq:energyerror}  to get
\begin{equation*}
\|u-u_0\|^2-(\sigma-\sigma_h,\tilde{\sigma})\lesssim h^{k+2}\|u\|_{k+2}\|\tilde\sigma\|_1.
\end{equation*}
Using \eqref{equ_hy2}, the integration by parts, and \eqref{eq:energyerror}, we have
\begin{align*}
(\sigma-\sigma_h,\tilde{\sigma})&=(\sigma-\sigma_h,\nabla_h(\tilde{u}-Q_0\tilde{u})) + (\sigma,\nabla_h(Q_0\tilde{u})) \\
&\quad + b_h(\sigma_h,Q_h\tilde{u})- \sum_{K\in\mathcal K_h}\langle\sigma_h\cdot\boldsymbol{n}, Q_0\tilde{u}\rangle_{\partial K} \\
&=(\sigma-\sigma_h,\nabla_h(\tilde{u}-Q_0\tilde{u})) + (\sigma,\nabla_h(Q_0\tilde{u})) - (f,Q_0\tilde{u}) \\
&\quad - \sum_{K\in\mathcal K_h}\langle\sigma_h\cdot\boldsymbol{n}, Q_0\tilde{u}\rangle_{\partial K} \\
&=(\sigma-\sigma_h,\nabla_h(\tilde{u}-Q_0\tilde{u})) + \sum_{K\in\mathcal K_h}\langle(\sigma-\sigma_h)\cdot\boldsymbol{n}, Q_0\tilde{u}-\tilde{u}\rangle_{\partial K} \\
&\lesssim h\|\sigma-\sigma_h\|_{0,h}\|\tilde{u}\|_2\lesssim h^{k+2}\|u\|_{k+2}\|\tilde{u}\|_2.
\end{align*}
Finally, we conclude \eqref{eq:uL2error} from the last two inequalities and the regularity \eqref{eq:dualregularity}.
\end{proof}

\section{Connections to Other Methods}\label{sec:connection}
In this section, we investigate some connections between our SDG method and other methods, such as Crouzeix-Raviart (CR) linear non-conforming finite element on simplices, stabilization-free weak gradient (WG) method, and stabilization-free non-conforming virtual element methods (VEM). One benefit of stabilization free is the local conservation property. In traditional WG and VEM, the stabilizer will destroy the local conservation; see Remark \ref{rm:conservation}. 

\subsection{Nonconforming linear element method on simplices}\label{subsec:conn2CR}
In this subsection, we take $k=0$, and assume each element $K\in\mathcal{K}_h$ is a simplex which is automatically true in two dimensions and can hold by further refining each face $F$ into $(d-1)$-dimensional simplices, see Fig. \ref{fig_refineF} as an example. 
\begin{figure}[htbp]
\label{fig_refineF}
\subfigure[An element $K\in\mathcal{K}_h$ in 3D.]{
\begin{minipage}[t]{0.45\linewidth}
\centering
\includegraphics*[width=3.25cm]{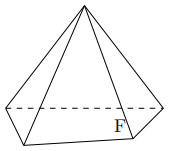}
\end{minipage}}
\subfigure[Refine $F$ into triangles.]
{\begin{minipage}[t]{0.45\linewidth}
\centering
\includegraphics*[width=3.25cm]{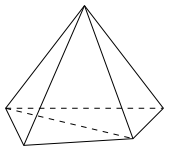}
\end{minipage}}
\caption{ Refine an element $K$ into simplices.}
\end{figure} 
Now on the refined mesh $\mathcal T_h$, we can define the CR nonconforming linear element space 
\begin{align*}
U_h^{\rm CR}:=\{v_h\in L^2(\Omega):&\, v_h|_T\in \mathbb P_1(T) \textrm{ for } T\in  \mathcal T_h, \,\int_F[v_h] \dd S =0  \textrm{ on } F\in \mathring{\mathcal{F}}_h^{T} \},
\end{align*}
and its subspace $\mathring{U}_h^{\rm CR} = \{v_h\in U_h^{\rm CR}, \int_F v_h \dd S = 0 \text{ on } F \in \mathcal{F}_h^{T}\cap \partial \Omega\}$. 
A function $v_h\in U_h^{\rm CR}$ is uniquely determined by the average of function on each face $F\in \mathcal F_h^T$. 

Introduce a connection operator $E_h: M_h\to U_h^{\rm CR}$ as follows: for $v_h=(v_0,v_b)\in M_h$, define $E_h v_h \in U_h^{\rm CR}$ such that
\begin{align*}
\int_F (E_hv_h) \dd S&= \int_F v_b \dd S \quad\quad \textrm{ on } F\in \mathcal{F}_h^{K},\\
\int_F (E_hv_h) \dd S &= \int_F v_0 \dd S \quad\;\;\; \textrm{ on } F\in \mathcal{F}_h^T \setminus \mathcal{F}_h^{K}.
\end{align*}

\begin{lemma}
Let $k=0$.
Assume each element $T\in\mathcal{T}_h$ is a simplex. For $v_h=\{ v_0, v_b \}\in M_h$,
we have 
\begin{equation}\label{equivWeakgradCR}
\nabla_w v_h=\nabla_h(E_hv_h).
\end{equation}  
\end{lemma}
\begin{proof}
For $\tau\in\Sigma_h^{-1}(K)$ and $K\in \mathcal K_h$, It follows from the integration by parts and the definition of $E_hv_h$ that
\begin{align}
\label{weakgradgradCR}
(\nabla_{w} v_h,\tau)_K &=(\nabla v_0, \tau)_K+\langle v_b-v_0,\tau\cdot\boldsymbol{n}\rangle_{\partial K} \\
\notag
&=\sum_{F\in \mathcal{F}_h^T(K)\setminus \partial K} \langle v_0, [\tau\cdot\boldsymbol{n}]\rangle_{F}+\langle v_b,\tau\cdot\boldsymbol{n}\rangle_{\partial K} \\
\notag
&=\sum_{F\in \mathcal{F}_h^T(K)\setminus \partial K} \langle E_hv_h,  [\tau\cdot\boldsymbol{n}]\rangle_{F}+\langle E_hv_h,\tau\cdot\boldsymbol{n}\rangle_{\partial K} \\
\notag
&=(\nabla_h(E_hv_h),\tau)_K.
\end{align}
Therefore, \eqref{equivWeakgradCR} is true.
\end{proof}

By \eqref{equivWeakgradCR}, the connection operator $E_h: M_h\to U_h^{\rm CR}$ is injective, but in general not bijective as $\dim M_h\leq\dim U_h^{\rm CR}$. In other words, comparing to use CR element on the sub-triangulation $\mathcal T_h$, $M_h$ has fewer DoFs. When each element $K\in\mathcal{K}_h$ is a triangle in two dimensions, $E_h$ is bijective, as $\dim M_h=\dim U_h^{\rm CR}$; cf. Fig.~\ref{fig:P0P1CR} (a). 
\begin{figure}[htbp]
\label{fig:P0P1CR}
\begin{centering}
\subfigure[Equivalence of $P_0$-$P_1$ SDG to CR element on barycentric refinement of triangles.]{
\begin{minipage}[t]{0.5\linewidth}
\centering
\includegraphics*[width=6.5cm]{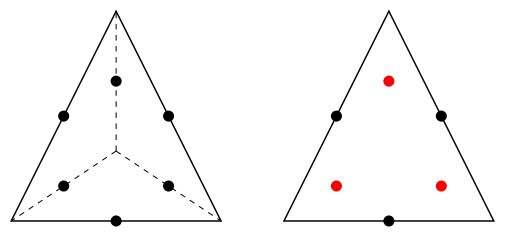}
\end{minipage}}
\qquad
\subfigure[$P_0$-$P_1$ SDG on a quadrilateral.]
{\begin{minipage}[t]{0.35\linewidth}
\centering
\includegraphics*[width=3.5cm]{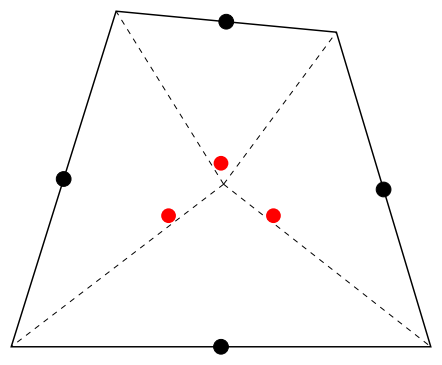}
\end{minipage}}
\caption{Connection of $P_0$-$P_1$ SDG and CR non-conforming linear finite elements.}
\end{centering}
\end{figure}


There are several works~\cite{Han1984,RannacherTurek1992,CaiDouglasSantosSheenEtAl2000,LinTobiskaZhou2005,ParkSheen2003,HuShi2005,AyusodeDiosLipnikovManzini2016,Wang2019} on extension of linear non-conforming finite element to quadrilaterals and the focus is a suitable definition of shape functions. Now by introducing a separate $u_0\in \mathbb P_1(K)$ as the shape function inside $K$ and using weak gradient instead of element-wise gradient, the hybridized $P_0$-$P_1$ SDG can be thought as a generalization of linear non-conforming finite element to general polytopes; cf. Fig.~\ref{fig:P0P1CR} (b). 

\begin{lemma}
Let $k=0$.
Assume each element $T\in\mathcal{T}_h$ is a simplex. For $v_h=\{v_0, v_b\}\in M_h^0$ and $K\in\mathcal K_h$,
we have 
\begin{equation}\label{computgradv0}
(\nabla v_0,\tau)_K=(\nabla_h(E_hv_h),\tau)_K \quad \forall~\tau\in{\rm span}\{\boldsymbol{t}_{\ell}\chi_{T_{\boldsymbol{t}_{\ell}}}, \ell=1,\ldots,d\},
\end{equation}  
where $T_{\boldsymbol{t}_\ell}\in\mathcal{T}_h(K)$ is the subelement with tangent vector $\boldsymbol{t}_\ell$ chosen in \eqref{equ_edge_vector_property}. Then on each $K\in\mathcal K_h$,
\begin{equation}\label{explicitcomputgradv0}
\nabla v_0=\sum_{\ell=1}^d\hat{\boldsymbol{t}}_{\ell}(\boldsymbol{t}_{\ell}\cdot\nabla((E_hv_h)|_{T_{\boldsymbol{t}_{\ell}}})).
\end{equation}
Here $\{\hat{\boldsymbol{t}}_\ell, \ell=1,...,d\}$ be the dual bases of $\{\boldsymbol{t}_\ell, \ell=1,...,d\}$.
\end{lemma}
\begin{proof}
For $\tau\in{\rm span}\{\boldsymbol{t}_{\ell}\chi_{T_{\boldsymbol{t}_{\ell}}}, \ell=1,\ldots,d\}$, it holds $\tau \cdot \boldsymbol{n}|_{\partial K}=0$. Then \eqref{computgradv0} following from 
\eqref{weakgradgradCR}.
\end{proof}

By \eqref{explicitcomputgradv0}, we can compute $\nabla v_0$ from $E_hv_h$, then recover $v_0$.

\begin{lemma}
Let $u_h\in M_h^0$ be the solution of the weak Galerkin method \eqref{eqn_WG_scheme} with $k=0$. 
Assume each element $T\in\mathcal{T}_h$ is a simplex.
Set $u_h^{CR}=E_hu_h$, then $u_h^{CR}\in \mathring{U}_h^{\rm CR}$ is the solution of the nonconforming linear element method
\begin{equation}\label{eq:CRscheme}
(\nabla u_h^{\rm CR}, \nabla_h v)=(f, v_0), \quad\forall~v\in \mathring{U}_h^{\rm CR},
\end{equation}  
where $v_0$ is computed from \eqref{explicitcomputgradv0} up to a constant.
\end{lemma}
\begin{proof}
It is an immediate result of \eqref{equivWeakgradCR}.
\end{proof}

\subsection{Stabilization free weak Galerkin methods}\label{sec:WG}
Recall that scheme \eqref{equ_hy} is given by
\begin{equation}
\begin{aligned}
(\sigma_h,\tau_h)-(\nabla_{w} u_h,\tau_h)&=0,&\;&\forall~\tau_h\in \Sigma^{-1}_h,
\\
-(\nabla_{w} v_h,\sigma_h)&=(-f,v_0), &\;&\forall~v_h\in M_h^0.
\end{aligned}
\end{equation}
Using the fact that $\nabla_{w} u_h\in \Sigma^{-1}_h$, the first equation gives that $\sigma_h=\nabla_{w} u_h$. Then scheme \eqref{equ_hy} becomes: Find $u_h\in M_h^0$, satisfying
\begin{equation}
(\nabla_w u_h,\nabla_w v)=(f,v_0), \quad\forall~v\in M_h^0.\label{eqn_WG_scheme}
\end{equation}

Note that scheme (\ref{eqn_WG_scheme}) can be treated as a weak Galerkin method without stabilizers. 
The inf-sup condition (\ref{eqn_inf_sup_WG}) will imply the coercivity of \eqref{eqn_WG_scheme}.

\begin{lemma}
We have the norm equivalence
\begin{equation}\label{eq:weakgradnormequiv}
\|\nabla_wu_h\|\eqsim \|u_h\|_{1,h}\quad\forall~u_h\in M_h^0.  
\end{equation}
Then the weak Galerkin method \eqref{eqn_WG_scheme} is well-posed.
\end{lemma}
\begin{proof}
The discrete inf-sup \eqref{eqn_inf_sup_WG} is equivalent to 
\begin{equation*}
\|u_h\|_{1,h}\lesssim \|\nabla_wu_h\| \quad\forall~u_h\in M_h^0.  
\end{equation*}
The other side of \eqref{eq:weakgradnormequiv} follows from the definition \eqref{eq:wg} of the weak gradient and the Cauchy-Schwarz inequality.
\end{proof}


Recently there is a trend to derive stabilization free polytopal element methods~\cite{FirstSFVEM,YeZhang2020,ChenHuangWei2024}. For WG, the common feature is to enrich the space of computing the weak gradient. In~\cite{YeZhang2020,AlTaweelWang2020,AlTaweelWang2020a,YeZhang2021,YeZhang2021a,YeZhang2021b,AlTaweelWangYeZhang2021}, the degree of polynomial for $\nabla_w u_h$ is not less than that for $u_0$. For our stabilization free scheme \eqref{eqn_WG_scheme}, the degree of polynomial for $\nabla_w u_h$ is one-order lower than that for $u_0$ but a piecewise polynomial space is used. Similar idea can be applied to VEM as discussed in the next subsection.

\subsection{Stabilization free non-conforming virtual element methods}
Introduce the $H^1$-nonconforming virtual element, for $k\geq 0$,
\[
W_{k+1}(K):=\left\{u\in H^1(K): \Delta u\in \mathbb P_{k+1}(K), \;\partial_n u|_F\in\mathbb P_{k}(F), \quad\forall~F\in\partial K\right\}.
\]
It is determined by the DoFs:
\begin{align}
\frac{1}{|F|}\int_F v\, q \dd S, & \quad q\in\mathbb P_{k}(F), F\in\partial K, \label{H1dof1}\\
\frac{1}{|K|}\int_K v\, q \dd x, & \quad q\in\mathbb P_{k+1}(K). \label{H1dof2}
\end{align}
Define the $H^1$-nonconforming virtual element space 
\begin{align*}
W_h:=\{v_h\in L^2(\Omega):&\, v_h|_K\in W_{k+1}(K) \textrm{ for } K\in  \mathcal K_h, \,Q_b([v_h]|_F)=0  \textrm{ on } F\in \mathcal{F}_h^K \}.
\end{align*}

Let $Q_M: W_h\to M_h$ be the $L^2$-projection operator defined as follows: for $v\in W_h$, $Q_Mv=\{Q_0v, Q_bv\}\in M_h$. Indeed, $Q_M$ is bijective: $v_0 = Q_0 v$ can be computed by DoF \eqref{H1dof2} and $v_b=Q_b v$ by DoF \eqref{H1dof1}. Then \eqref{eqn_WG_scheme} can be rewritten as the following virtual element method: Find $u_h\in W_h$ such that
\begin{equation}
(\nabla_w (Q_Mu_h),\nabla_w (Q_Mv))=(f,Q_0v), \quad\forall~v\in W_h.\label{eqn_VEM_scheme}
\end{equation}

\begin{lemma}
We have the norm equivalence
\begin{equation*}
\|\nabla_w (Q_Mu_h)\|^2\eqsim \sum_{K\in\mathcal K_h}\|\nabla u_h\|_{K}^2\quad\forall~u_h\in W_h.  
\end{equation*}
Then the virtual element method \eqref{eqn_VEM_scheme} is well-posed.
\end{lemma}
\begin{proof}
Thanks to \eqref{eq:weakgradnormequiv}, it suffices to prove
\begin{equation}\label{eq:vemnormequiv0}
\|Q_Mu_h\|_{1,h}^2\eqsim \sum_{K\in\mathcal K_h}\|\nabla u_h\|_{K}^2\quad\forall~u_h\in W_h.  
\end{equation}
Applying the inverse inequality and the estimates of $Q_0$ and $Q_b$, we acquire
\begin{equation*}
\|Q_Mu_h\|_{1,h}^2\lesssim \sum_{K\in\mathcal K_h}\|\nabla u_h\|_{K}^2\quad\forall~u_h\in W_h.  
\end{equation*}

For the other side of \eqref{eq:vemnormequiv0},
\begin{align*}
\|\nabla_h(u_h-Q_0u_h)\|_K^2&=(\nabla_h(u_h-Q_0u_h),\nabla_h(u_h-Q_0u_h))_K\\
&= (\partial_n(u_h-Q_0u_h),Q_bu_h-Q_0u_h)_{\partial K} \\
&\leq \|\partial_n(u_h-Q_0u_h)\|_{\partial K}\|Q_bu_h-Q_0u_h\|_{\partial K}.
\end{align*}
By employing the bubble function technique as in~\cite{ChenHuang2020,Huang2020}, we have
\begin{equation*}
\|\partial_n(u_h-Q_0u_h)\|_{\partial K}\lesssim h_K^{-1/2}\|\nabla_h(u_h-Q_0u_h)\|_K.
\end{equation*}
The combination of the last two inequalities yields
\begin{equation*}
\|\nabla_h(u_h-Q_0u_h)\|_K\lesssim h_K^{-1/2}\|Q_bu_h-Q_0u_h\|_{\partial K}.
\end{equation*}
Therefore, \eqref{eq:vemnormequiv0} holds from the triangle inequality.
\end{proof}

That is, by enriching the space for computing the projection of the gradient to piecewise polynomial, we get a stabilization free non-conforming virtual element method. 
A different stabilization free non-conforming virtual element method is developed in~\cite{ChenHuangWei2024} recently.


\section{Numerical Examples}\label{sec:numerexam}
In this section, we present three numerical experiments to describe our methods given in this paper with $k=0$ and $d=2$. 
All numerical experiments were developed based on the MATLAB package $i$FEM~\cite{Chen_iFem}.

\subsection{Implementation detail}
Recall the weak gradient of $u_h = \{u_0, u_b\}$ is defined as
\begin{equation}\label{eq:wgc}
(\nabla_{w} u_h,\tau)_K =(\nabla u_0, \tau)_K+\langle u_b-u_0,\tau\cdot\boldsymbol{n}\rangle_{\partial K},\quad \forall~\tau\in\Sigma_h^{-1}(K).
\end{equation}

The global stiffness matrix is assembled by the local stiffness $A_K$ on each polygon $K\in\mathcal{K}_h$. As a concrete example, take $K$ as a pentagon. The weak function $u_h$ is composed by a linear combination of bases $\{\phi_b^i, \phi_0^j, i=1,...,5, j=1,2,3, \}$:
$$
\phi_b^i=\chi_{F_i}, \;i=1,...,5,\quad \phi_0^1= 1,\quad \phi_0^2= (x-x_K)/h_K,\quad \phi_0^3= (y-y_K)/h_K,
$$
where $(x_K, y_K)$ is the average of vertices of $K$ and $h_K = |K|^{1/2}$.

A weak gradient is expressed by a linear combination of bases $\{\boldsymbol{\zeta}_i,i=1,...,10\}=\{\boldsymbol{n}_i\chi_{T_i},  \boldsymbol{t}_i\chi_{T_i}, i=1,...,5\}$. The mass matrix of $\{\boldsymbol{\zeta}_i,i=1,...,10\}$ is $$M = {\rm diag}(|T_1|, \ldots, |T_5|, |T_1|, \ldots, |T_5|),$$
where $T_i$ is the triangle formed by $F_i$ and $x_K$. 
For general elliptic equation $\boldsymbol K^{-1}\sigma = \nabla u$, we only need to update $M$ by using $M_{\boldsymbol K^{-1}} = (\int_{K}\boldsymbol{\zeta}_i \boldsymbol K^{-1} \boldsymbol{\zeta}_j \dd x)$. 

Let $(x_{F_i}, y_{F_i})$, $i=1,...,5$ be the middle points of edges $F_i$. 
Then the weak gradients  $\nabla_w\phi_b^j$, $\nabla_w\phi_0^i$ are given by 
$$
[\nabla_w\phi_b^1,...,\nabla_w\phi_b^5,\nabla_w\phi_0^1,...,\nabla_w\phi_0^3 ]=M^{-1}[D_b, D_0],
$$
where $D_b$ and $D_0$ are obtained by setting $u_0 =\phi_0^i$ and $u_b =\phi_b^j$ on the right hand side of \eqref{eq:wgc}:
\begin{align*}
D_b &=\left[
\begin{array}{cccc}
|F_1| & 0 & \ldots  &0\\
0 & |F_2| & \ldots  &0\\
\vdots& \vdots & \ddots&\vdots \\
0 &0& \ldots&|F_5|\\
0&0&\ldots &0\\
\vdots & \vdots &  & \vdots \\
0&0&\ldots &0
\end{array}
\right],
\\
D_0&=\left[
\begin{array}{ccc}
-|F_1| &  \left [|T_1|\boldsymbol{n}_1^x+|F_1|(x_K-x_{F_1})\right ]/h_K& \left [ |T_1|\boldsymbol{n}_1^y+|F_1|(y_K-y_{F_1})\right ]/h_K\\[0.1in]
\vdots & \vdots & \vdots  \\
-|F_5| & \left [|T_5|\boldsymbol{n}_5^x+|F_5|(x_K-x_{F_5})\right ]/h_K& \left [|T_5|\boldsymbol{n}_5^y+|F_5|(y_K-y_{F_5})\right ]/h_K\\[0.1in]
0 & |T_1|\boldsymbol{t}_1^x/h_K& |T_1|\boldsymbol{t}_1^y/h_K\\[0.1in]
\vdots & \vdots & \vdots \\
0 & |T_5|\boldsymbol{t}_5^x/h_K& |T_5|\boldsymbol{t}_5^y/h_K
\end{array}
\right].
\end{align*}
The local stiffness matrix is computed by
$$
A_K = [D_b,D_0]^{\intercal}M^{-1}[D_b,D_0].
$$
For $\boldsymbol K^{-1}\sigma = \nabla u$, it becomes 
$$
A_K = [D_b,D_0]^{\intercal}M_{\boldsymbol K^{-1}}^{-1}[D_b,D_0].
$$
Notice that we are using the harmonic average of $\boldsymbol K$ in the primal formulation, which is more appropriate in heterogeneous media.

\subsection{Triangle mesh}
We consider the unit square domain $\Omega=(0,1)^2$ and apply a triangular mesh discretization $\mathcal{K}_h$, see Fig. \ref{fig:triangulation} (a). By taking the barycenter of each triangle, we obtain a refined triangular mesh $\mathcal{T}_h$, see Fig. \ref{fig:triangulation} (b) for mesh size $h=0.25$. In this example, we focus on the problem (\ref{eqn_WG_scheme}) for the lowest order case $k=0$.
We consider the exact solution as
$$
u=\cos(\pi x) \cos(\pi y),
$$
and the right-hand side is correspondingly given by
$f=2\pi^2 \cos(\pi x) \cos(\pi y).$

\begin{figure}[htbp]
\label{fig:triangulation}
\subfigure[A triangle partition $\mathcal{K}_h$.]{
\begin{minipage}[t]{0.45\linewidth}
\centering
\includegraphics*[width=4.8cm]{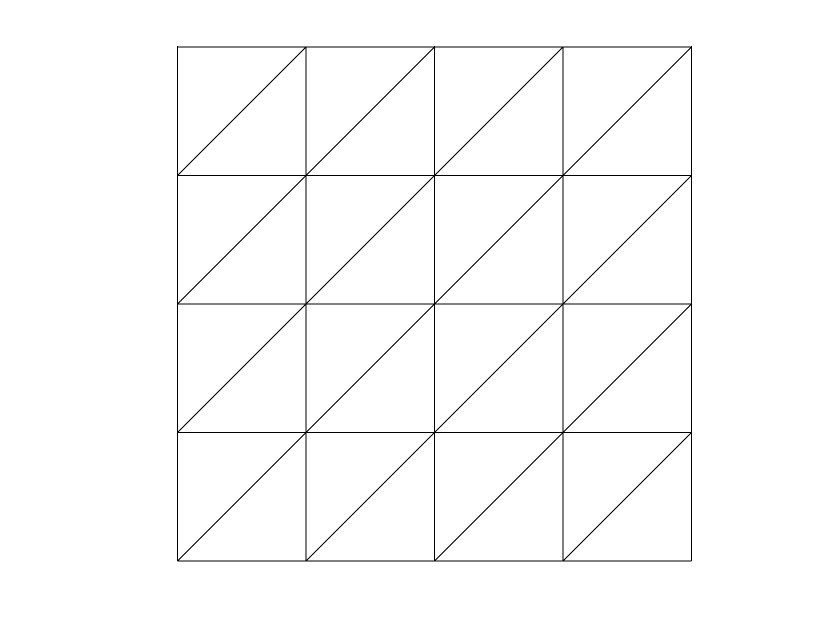}
\end{minipage}}
\subfigure[Triangulation $\mathcal{T}_h$ for $\mathcal{K}_h$.]
{\begin{minipage}[t]{0.45\linewidth}
\centering
\includegraphics*[width=4.8cm]{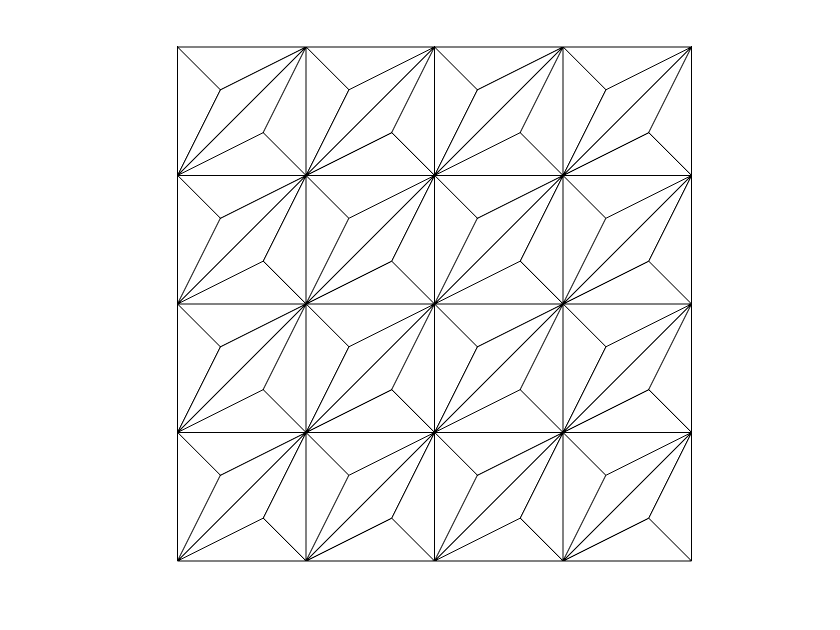}
\end{minipage}}
\caption{Partitions of $\Omega$ for $h=0.25$.}
\end{figure}

As described in the subsection \ref{subsec:conn2CR}, the numerical solution 
$u_h$ of (\ref{eqn_WG_scheme}) is equivalent to the Crouzeix-Raviart (CR) element solution of (\ref{eq:CRscheme}).
Then mesh sizes $h$ are taken as $h=\frac{1}{32}, \frac{1}{128}, \frac{1}{512}, \frac{1}{2048}, \frac{1}{8192}$.  The number of polygons in the mesh is denoted by $N_K$.
Table \ref{Ex1_P0P1CR} shows the error between the numerical and exact solutions in the $L^2$-norm $\|u-u_h\|$ and $H^1$-seminorm $\|\nabla u - \nabla_w u_h\| = \| \sigma - \sigma_h \|$, along with the corresponding convergence rates, which exhibit second-order and first-order convergence, respectively.

\begin{table}[!ht]
  \centering
  \caption{Example 1 (Triangle mesh):
    Error and the convergence rate. }
    \renewcommand{\arraystretch}{1.125}
    \resizebox{7.5cm}{!}{
  \begin{tabular}{@{}c c c c c c @{}}
    \toprule
  $ h $  &    $N_K$   &  $ \| \sigma - \sigma_h \| $  
  &  Rate   &   $ \|u-u_h\| $  &  Rate
\\
    \hline
2.500e-01 &    32  &  6.03095e-01 &  -- &  2.54911e-02 &  --\\
1.250e-01 &   128  &  3.02359e-01 &  1.00 &  6.33303e-03 &  2.01\\
6.250e-02 &   512  &  1.51292e-01 &  1.00 &  1.58159e-03 &  2.00\\
3.125e-02 &  2048  &  7.56601e-02 &  1.00 &  3.95307e-04 &  2.00\\
1.562e-02 &  8192  &  3.78319e-02 &  1.00 &  9.88212e-05 &  2.00\\
    \bottomrule
  \end{tabular}}
\label{Ex1_P0P1CR}
\end{table}


\subsection{Polygon mesh}
We examine the unit square domain $\Omega=(0,1)^2$ and use polygon meshes, see Fig. \ref{fig_polygonmesh} with 16 polygons. This example is centered on the lowest-order case $k=0$.
We consider the exact solution as
$$
u=\cos(\pi x) \cos(\pi y)-1,
$$
and the right-hand side is given accordingly.
The errors between the numerical solution $u_h$ of (\ref{eqn_WG_scheme}) and the exact solution $u$ are measured in the following norms:
\begin{align*}
\|u-u_h\|_{1,h}&=\left(\sum_{K\in\mathcal{K}_h}(\|\nabla u-\nabla u_0\|_K^2+h^{-1}_K\|u_0-u_b\|_{\partial K}^2)\right)^{\frac12},\\
\|u-u_0\| &=\left(\sum_{K\in\mathcal{K}_h}\|u-u_0\|_K^2\right)^{\frac12}, \\
\|\sigma-\sigma_h\|_{0,h}&=\|\nabla u-\nabla_w u_h\|_{0,h}\\
&=\left( \sum_{K\in\mathcal{K}_h}\|\nabla u-\nabla_w u_h\|_K^2
+\sum_{K\in \mathcal{K}_h}h_K\|(\nabla u-\nabla_w u_h) \cdot \boldsymbol{n}\|^2_{\partial K}  \right)^{\frac12}.
\end{align*}
The middle point of edges in $\mathcal T_h$ are used as the quadrature points which is exact for the quadratic integrand. The error and convergence order are given in Table \ref{Ex2_P0P1poly}.

\begin{figure}[htbp]
\label{fig_polygonmesh}
\subfigure[A polygon partition $\mathcal{K}_h$.]{
\begin{minipage}[t]{0.45\linewidth}
\centering
\includegraphics*[width=4.8cm]{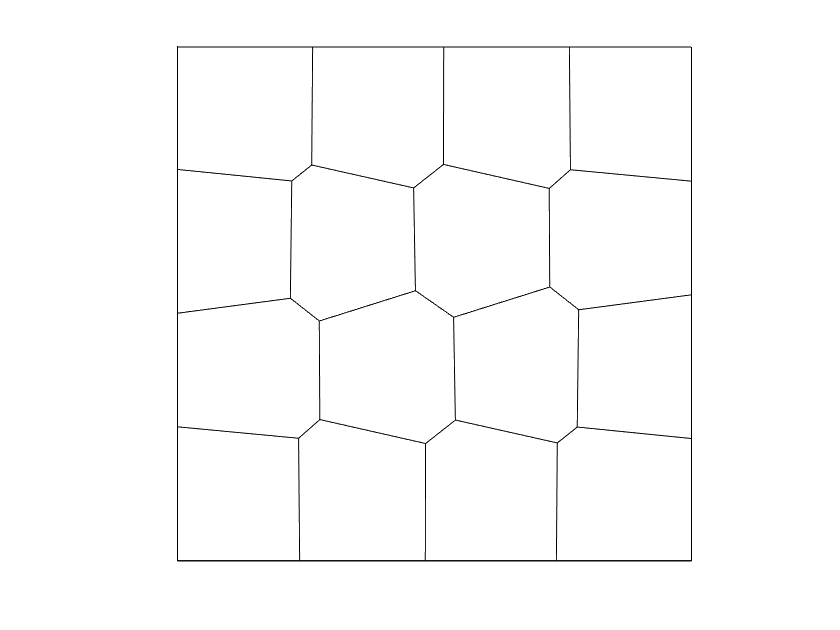}
\end{minipage}}
\subfigure[Triangulation $\mathcal{T}_h$ for $\mathcal{K}_h$.]
{\begin{minipage}[t]{0.45\linewidth}
\centering
\includegraphics*[width=4.8cm]{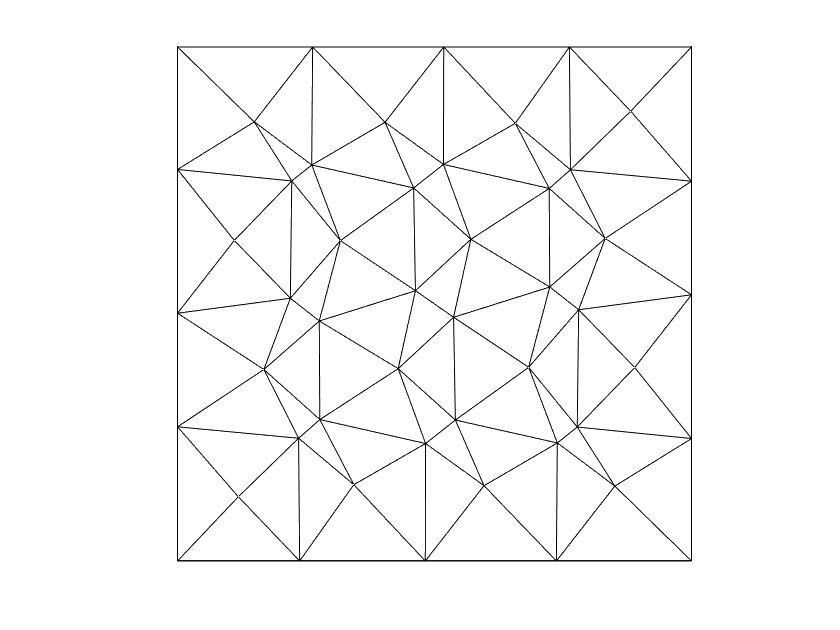}
\end{minipage}}
\caption{A polygonal mesh of $\Omega$ with $N_K=16$.}
\end{figure}


\begin{table}[!ht]
  \centering
  \caption{Example 2 (Polygon mesh):
    Error and the convergence rate.}
    \renewcommand{\arraystretch}{1.125}
    \resizebox{11cm}{!}{
  \begin{tabular}{@{}c c c c c c c c @{}}
    \toprule
  $ h $   &    $ N_K$ & $\|u-u_h\|_{1,h} $ & Rate & $ \|u-u_0\| $ & Rate & $ \|\sigma-\sigma_h\|_{0,h} $ & Rate 
    \\
    \hline
1.250e-01 &     64 &  4.18367e-01 &  -- &  9.86917e-03 &  -- &  8.15496e-01 &  --\\
6.250e-02 &    256 &  2.05838e-01 &  1.02 &  2.50172e-03 &  1.98 &  4.11331e-01 &  0.99\\
3.125e-02 &   1024 &  1.03069e-01 &  1.00 &  6.24662e-04 &  2.00 &  2.05950e-01 &  1.00\\
1.562e-02 &   4096 &  5.15408e-02 &  1.00 &  1.58685e-04 &  1.98 &  1.03652e-01 &  0.99\\
7.069e-03 &  20014 &  2.56836e-02 &  1.00 &  3.98802e-05 &  1.99 &  5.15632e-02 &  1.01\\
    \bottomrule
  \end{tabular}}
\label{Ex2_P0P1poly}
\end{table}

\subsection{Local conservation}   
We take an example from~\cite{sun2009locally} to illustrate the local conservation property. Consider Darcy's problem of unit flow in horizontal direction with hydraulic conductivity coefficient $\boldsymbol{K}$: Find $u$ satisfying
\begin{equation*}
  -\nabla\cdot(\boldsymbol{K}\nabla u)=f, \quad\text{in}\;\Omega,
\end{equation*}
with $\Omega=(0,1)^2$, and the boundary conditions
$$
u=1,\quad\text{left},\quad u=0,\quad\text{right},\quad -\boldsymbol{K}\nabla u\cdot\boldsymbol{n}=0,\quad\text{elsewhere}.
$$
The conductivity is a diagonal tensor $\boldsymbol{K}=\kappa\boldsymbol{I}$ with $\kappa$ being $10^{-3}$ in $(\frac38,\frac58)\times(\frac14,\frac34)$ and $1$ elsewhere, see Fig. \ref{fig:poorlypermreg} (a).  
The right-hand side function is given by $f=0$. A uniform $32\times 32$ square mesh is used. In observation of Fig. \ref{fig:poorlypermreg} (b), the fluid flows from left to right, and the direction of the arrows indicates that the fluid will surround the low-permeability region.

\begin{figure}[htbp]
\label{fig:poorlypermreg}
\subfigure[Value $\kappa$ of conductivity $\boldsymbol{K}$.]{
\begin{minipage}[t]{0.315\linewidth}
\centering
\includegraphics*[width=4.0cm,height=3.2cm]{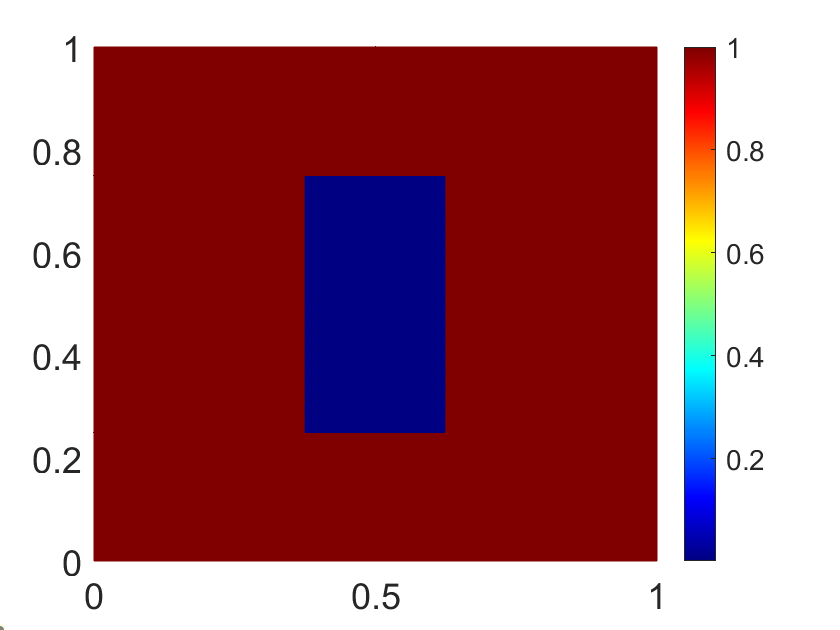}
\end{minipage}}
\subfigure[Computed $u_h$ and $\sigma_h$/$|\sigma_h|$.]
{\begin{minipage}[t]{0.325\linewidth}
\centering
\includegraphics*[width=3.5cm,height=3.0cm]{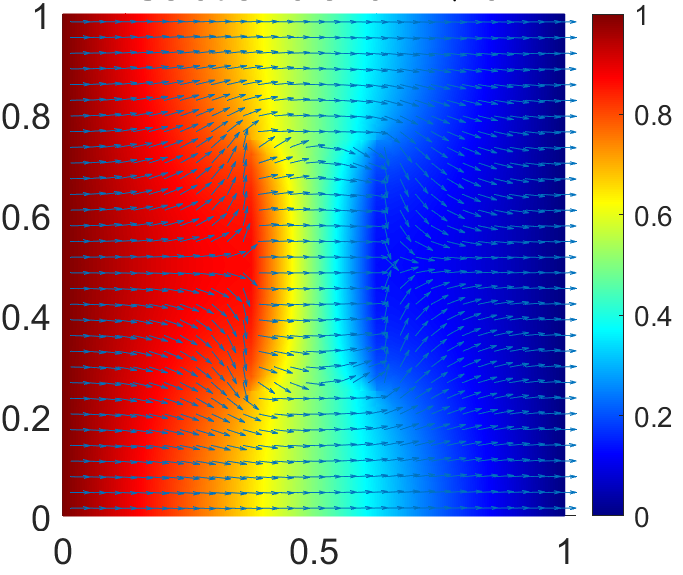}
\end{minipage}}
\subfigure[The local conservation residual.]
{\begin{minipage}[t]{0.315\linewidth}
\centering
\includegraphics*[width=3.6cm,height=3.125cm]{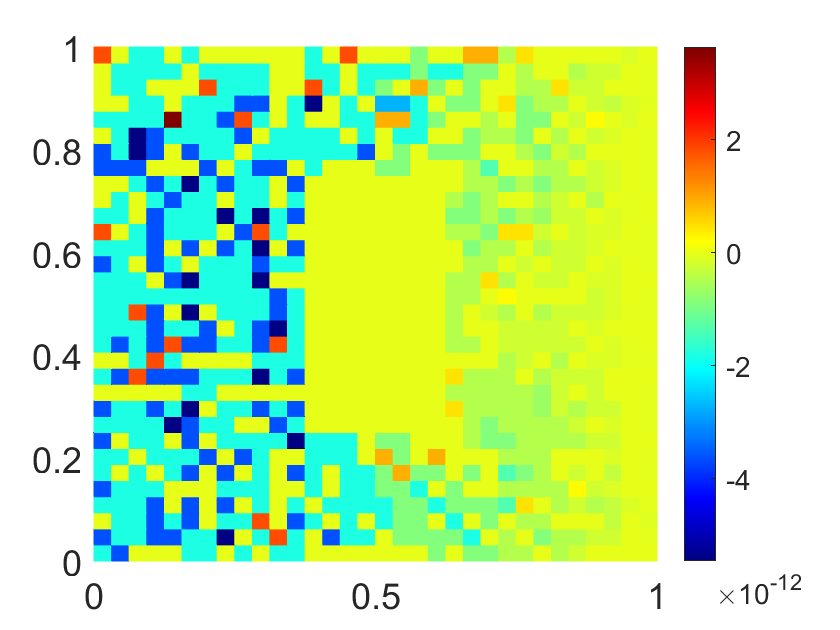}
\end{minipage}}
\caption{Example 3: The pressure $u$ and velocity $-\boldsymbol{K}\nabla u$.}
\end{figure}

With the Darcy velocity $\sigma:=-\boldsymbol{K}\nabla u$, the numerical scheme is given by
\begin{equation}
\begin{aligned}\label{eq:Ex3}
(\boldsymbol{K}^{-1}\sigma_h,\tau_h)+(\nabla_{w} u_h,\tau_h)&=0,&\;&\forall~\tau_h\in \Sigma^{-1}_h,
\\
(\sigma_h, \nabla_{w} v_h)&=(-f,v_0), &\;&\forall~v_h\in M_h^0.
\end{aligned}
\end{equation}
Taking $v_h=\{\chi_K, 0\}$ in the second equation of (\ref{eq:Ex3}), we have on each element $K\in\mathcal{K}_h$ the local conservation
$$
\int_{K}f \dx-\int_{\partial K}\sigma_h\cdot\boldsymbol{n} \dd S=0.
$$
In Fig.~\ref{fig:poorlypermreg}~(c), we display the residual $(\int_{K}f \dd x-\int_{\partial K}\sigma_h\cdot\boldsymbol{n} \dd S)/|K|$ which is in the order of $10^{-12}$. The importance of local conservation is further illustrated by examples in~\cite{sun2009locally}.


%

\end{document}